\documentclass[letterpaper,12pt]{amsart}
\usepackage[margin=1in]{geometry}
\usepackage{graphicx, epigraph}
\usepackage{amsmath,amsthm,amssymb,mathtools,amsfonts,mathrsfs,stmaryrd, latexsym}
\usepackage{xspace,xcolor}
\usepackage[breaklinks,colorlinks,citecolor=teal,linkcolor=teal,urlcolor=teal,pagebackref,hyperindex]{hyperref}
\usepackage[alphabetic]{amsrefs}
\usepackage{tikz-cd}
\usepackage{tikz}
\usetikzlibrary{arrows.meta,positioning,calc,decorations.pathmorphing,fit,shapes.geometric}
\usepackage[all]{xy}
\usepackage{url}
\usepackage[shortlabels]{enumitem}
\usepackage{booktabs}
\usepackage{array}
\usepackage{longtable}
\usepackage{cleveref}

\usepackage[nodisplayskipstretch]{setspace}
\newtheorem{thm}{Theorem}[section]
\newtheorem{prop}[thm]{Proposition}

\newtheorem{lem}[thm]{Lemma}
\newtheorem{cor}[thm]{Corollary}

\theoremstyle{definition}

\newtheorem{rmk}[thm]{Remark}

\newtheorem*{nota*}{Notation}

\newcommand{\Scal}{\mathcal S}
\newcommand{\Ecal}{\mathcal E}
\newcommand{\id}{\mathbf 1}
\newcommand{\C}{\mathbb C}
\newcommand{\Q}{\mathbb Q}
\newcommand{\Z}{\mathbb Z}

\newcommand{\PP}{\mathbb P}
\newcommand{\cC}{\mathcal C}

\newcommand{\bM}{\overline{\mathcal M}}
\newcommand{\vir}{\mathrm{vir}}

\newcommand{\CR}{\mathrm{CR}}

\newcommand{\age}{\mathrm{age}}
\newcommand{\ev}{\mathrm{ev}}
\newcommand{\Aut}{\mathrm{Aut}}

\newcommand{\Part}{\mathrm{Part}}

\newcommand{\scrC}{\mathscr C}

\newcommand{\vars}{\varsigma}

\newcommand{\OrbCap}{\mathrm{Cap}}
\newcommand{\vac}{|0\rangle}
\newcommand{\rank}{\operatorname{rank}}

\newcommand{\thickslash}{\mathbin{\!\!\pmb{\fatslash}}}

\DeclareMathOperator{\Tree}{Tree}

\DeclareMathOperator{\Spec}{Spec}

\numberwithin{equation}{section}
\newif\ifshowdraftnotes
\showdraftnotestrue

\title[Virasoro Constraints for Orbifold Curves]
{Virasoro Constraints for Orbifold Curves}

\author[Hu]{Jianxun Hu}
\address{School of Mathematics, Sun Yat-sen University, Guangzhou 510275, China}
\email{stsjxhu@mail.sysu.edu.cn}

\author[Jin]{Pengrong Jin}
\address{Department of Mathematics, Southern University of Science and Technology, Shenzhen, Guangdong 518055, China}
\email{jinpr@mail.sustech.edu.cn}

\author[Ke]{Hua-Zhong Ke}
\address{School of Mathematics, Sun Yat-sen University, Guangzhou 510275, China}
\email{kehuazh@mail.sysu.edu.cn}

\author[Lu]{Jingwei Lu}
\address{School of Mathematics, Sun Yat-sen University, Guangzhou 510275, China}
\email{lujw28@mail2.edu.cn}

\author[Tseng]{Hsian-Hua Tseng}
\address{Department of Mathematics\\ Ohio State University\\ 231 West 18th Ave.\\ Columbus,  OH 43210\\ USA}
\email{hhtseng@math.ohio-state.edu}

\author[Wu]{Longting Wu}
\address{Department of Mathematics \& SUSTech International Center for Mathematics, Southern University of Science and Technology,
Shenzhen, Guangdong 518055, China}
\email{wult@sustech.edu.cn}

\date{}

\begin{document}

\begin{abstract}
We prove the Virasoro constraints for the relative Gromov--Witten theory of all smooth projective effective orbifold curves, allowing relative conditions at ordinary points. The absolute Jiang--Tseng Virasoro conjecture for smooth projective effective orbifold curves follows as a corollary.
\end{abstract}

\maketitle

\begingroup
\small
\tableofcontents
\endgroup

{\large \part{Setting, main theorem, and reduction}}

\section{Introduction}
\label{sec:introduction}

\subsection{Background}
Virasoro constraints are a system of recursive differential equations for descendant Gromov--Witten invariants.  They were predicted in the physics literature and were formulated mathematically for smooth projective varieties by Eguchi--Hori--Xiong and Katz. The constraints are expressed by operators \(L_k\), \(k\ge -1\), acting on the total descendant potential and satisfying the Virasoro commutation relations
\[
   [L_m,L_n]=(m-n)L_{m+n}.
\]
They encode, in a uniform way, the string equation, the dilaton/divisor-type equation, and infinitely many higher descendent recursions.

For semisimple Gromov--Witten theories, the work of Givental \cite{Giv5} \cite{Giv6} and Teleman \cite{Tel} gives a powerful tool to all-genus reconstruction and implies many Virasoro-type consequences.  However, general all-genus proofs in non-semisimple situations are comparatively rare, see \cite{OP3}, \cite{CGT}, \cite{Tseng26} for some established cases. 

The Virasoro conjecture has a natural extension to Gromov--Witten theory of a smooth proper Deligne--Mumford stack, spelled out explicitly by Jiang--Tseng \cite{JT}.  Their operators are defined on Chen--Ruan cohomology and include the age grading and the Chen--Ruan pairing.  In dimension one, these operators acquire a particularly transparent form: ordinary curve contributions are supplemented by age-shifted linear terms and by a new twisted-sector quadratic term.  

The purpose of the present paper is to prove Virasoro constraints for orbifold curves in the relative setting.

\subsection{Orbifold curves}
We work with smooth projective effective orbifold curves over $\C$. Such a curve can be written as a multi-root stack
\begin{equation*}
  \cC=C\bigl[\sqrt[r_1]{p_1},\ldots,\sqrt[r_s]{p_s}\bigr]
\end{equation*}
over its coarse curve $C$. Equivalently, \'etale locally near \(p_i\), if \(x_i\) is a local coordinate on \(C\) with \(p_i=\{x_i=0\}\), then \(\cC\) is locally of the form
\[
  [\Spec \mathbb{C}[z_i]/\mu_{r_i}],  \qquad x_i=z_i^{r_i},
\]
where \(\mu_{r_i}\) acts by \(z_i\mapsto \zeta_i z_i\) and \(\zeta_i\) is an \(r_i\)-th root of unity. Thus \(\cC\to C\) is an isomorphism away from \(p_1,\ldots,p_s\), while the point over \(p_i\) has stabilizer \(\mu_{r_i}\).

The inertia stack \(I\cC\) parametrizes pairs \((x,g)\), where \(x\in \cC\) and \(g\in\Aut_{\cC}(x)\).  It has the untwisted component \(\cC\) and, over each orbifold point \(p_i\), twisted components indexed by the nontrivial elements of \(\mu_{r_i}\).  We denote by
\[
  \phi_{i,a},\qquad 1\le a\le r_i-1,
\]
the Chen--Ruan class of the component corresponding to \(\zeta^a_i\).  Its age is
\[
  \age(\phi_{i,a})=\frac{a}{r_i}.
\]
With the standard unnormalized basis, the Chen--Ruan pairing satisfies
\[
  (\phi_{i,a},\phi_{j,b})_{\CR}=0
  \quad\text{unless }i=j\text{ and }a+b=r_i,
\]
and
\[
  (\phi_{i,a},\phi_{i,r_i-a})_{\CR}=\frac{1}{r_i}.
\]

The untwisted cohomology is the ordinary cohomology of the coarse curve.  We use the basis
\[
  \id,\quad \omega,\quad
  \alpha_1,\ldots,\alpha_g,\quad
  \beta_1,\ldots,\beta_g,
\]
where \(\omega\) is the point class of $C$, and \(\{\alpha_j,\beta_j\}\) is a symplectic basis of \(H^1(C)\).  Thus
\[
 H^*_{\CR}(\cC)=H^*(C)\oplus
 \bigoplus_{i=1}^s\bigoplus_{a=1}^{r_i-1}\C\phi_{i,a}.
\]

\subsection{Relative orbifold Gromov--Witten invariants}\label{section-relativeorbifoldGW}

Let
\[
  Q=\{q_1,\ldots,q_m\}\subset \cC
\]
be a set of ordinary, non-orbifold points, disjoint from \(p_1,\ldots,p_s\).  We write \(\cC/Q\) for the relative orbifold target.  A relative condition at \(q_j\) is an ordinary partition
\[
  \eta^j=(\eta^j_1,\ldots,\eta^j_{\ell_j}).
\]
For descendant insertions \(\tau_{k_u}(\gamma_u)\), \(\gamma_u\in H^*_{\CR}(\cC)\), the connected relative orbifold invariants are 
\begin{equation}\label{eqn:relorb-intro}
\left\langle
  \prod_{a=1}^{n}\tau_{k_
a}(\gamma_
a)
  \ \middle|\ \eta^1,\ldots,\eta^m
\right\rangle^{\circ,\,\cC/Q}_{g,d}
=\int_{[\bM_{g,n}(\cC/Q,\eta^1,\ldots,\eta^m)]^{\vir}}
\prod_{a=1}^{n}\bar\psi_
a^{k_
a}\ev_
a^*(\gamma_
a).
\end{equation}
Here, the moduli space
$$\bM_{g,n}(\cC/Q,\eta^1,\ldots,\eta^m)$$
parameterizes connected, genus $g$, $n$-pointed relative orbifold stable maps of the pair $(\cC,Q)$. Relative markings over $q_j$ are required to have contact orders prescribed by the partition $\eta^j$. There is a natural morphism
\[\pi:\bM_{g,n}(\cC/Q,\eta^1,\ldots,\eta^m)\longrightarrow \bM_{g,n}(C/Q,\eta^1,\ldots,\eta^m)\]
by forgetting the orbifold structures. The coarse cotangent classes $\bar\psi_a$ are pulling back from $\bM_{g,n}(C/Q,\eta^1,\ldots,\eta^m)$.

The superscript $\circ$ in \eqref{eqn:relorb-intro} stands for connected invariants, and for disconnected invariants, we use the superscript $\bullet$.

Let
\[
t^0_\ell,\quad t^1_\ell,\quad
s^j_\ell,\quad \bar s^j_\ell,\quad
x^{i,a}_\ell
\qquad(\ell\ge0)
\]
be the variables associated respectively to
\[
\tau_\ell(\id),\quad \tau_\ell(\omega),\quad
\tau_\ell(\alpha_j),\quad \tau_\ell(\beta_j),\quad
\tau_\ell(\phi_{i,a}).
\]
The variables $s^j_\ell,\bar s^j_\ell$ are odd and supercommute.
All derivatives with respect to these variables are left superderivatives.
We further set
\[\begin{aligned}
\xi=\sum_{\ell\ge0}\bigl(t^0_\ell\tau_\ell(\id)+t^1_\ell\tau_\ell(\omega)\bigr)
 +\sum_{j=1}^g\sum_{\ell\ge0}
 \bigl(s^j_\ell\tau_\ell(\alpha_j)+\bar s^j_\ell\tau_\ell(\beta_j)\bigr)
+\sum_{i=1}^{s}\sum_{a=1}^{r_i-1}\sum_{\ell\ge0}
 x^{i,a}_\ell\tau_\ell(\phi_{i,a}).
\end{aligned}\]
The disconnected relative potential of $\cC/Q$ with fixed relative data is defined by
\[
  Z^{\cC/Q}_d[\eta^1,\ldots,\eta^m]\coloneqq \sum_{n\geq 0}\frac{1}{n!}\left\langle
 \xi^n
  \ \middle|\ \eta^1,\ldots,\eta^m
\right\rangle^{\bullet,\,\cC/Q}_{d}.
\]
The genus is suppressed because it is determined by the dimension constraint in the disconnected generating function. When $m\neq 0$, the degree can also be suppressed since it is determined by the sum of each partition.
When $m=0$, we use the same notation to denote the disconnected potential of the absolute orbifold theory of $\cC$.

\subsection{Virasoro operators}\label{subsec:virasoro-operators}
The orbifold logarithmic Euler characteristic of $(\cC,Q)$ is
\begin{equation*}
  \chi_{\log}(\cC,Q)
  =\int_{\cC}c_1(T\cC(-\log Q))=2-2g(C)-|Q|-
  \sum_{i=1}^{s}\left(1-\frac1{r_i}\right).
\end{equation*}
We put
\[
\theta_{i,a}=\frac{a}{r_i},
\qquad a^\vee=r_i-a.
\]
The relative orbifold Virasoro operators of the pair $(\cC,Q)$ can be obtained from the Jiang--Tseng construction \cite{JT} by replacing $c_1(T\cC)$ with $c_1(T\cC(-\log Q))$. So $\chi(\cC)$ will be replaced by $\chi_{\log}(\cC,Q)$. We write out these operators explicitly: 
\begin{align*}
L_{-1}={}&-\frac{\partial}{\partial t^0_0}
+\sum_{\ell\ge0}\Bigg(
 t^0_{\ell+1}\frac{\partial}{\partial t^0_\ell}
+t^1_{\ell+1}\frac{\partial}{\partial t^1_\ell}
+\sum_{j=1}^g\left(
 s^j_{\ell+1}\frac{\partial}{\partial s^j_\ell}
+\bar s^j_{\ell+1}\frac{\partial}{\partial \bar s^j_\ell}
\right)\nonumber\\
&\hspace{37mm}
+\sum_{i=1}^s\sum_{a=1}^{r_i-1}
 x^{i,a}_{\ell+1}\frac{\partial}{\partial x^{i,a}_\ell}
\Bigg)\nonumber\\
&+ t^0_0t^1_0
+\sum_{j=1}^g s^j_0\bar s^j_0
+\frac1{2}\sum_{i=1}^s\sum_{a=1}^{r_i-1}
\frac1{r_i}x^{i,a}_0x^{i,a^\vee}_0,
\end{align*}
and 
\begin{align*}
L_0={}&-\frac{\partial}{\partial t^0_1}
-\chi_{\log}(\cC,Q)\frac{\partial}{\partial t^1_0}
+\sum_{\ell\ge0}\Bigg(
\ell t^0_\ell\frac{\partial}{\partial t^0_\ell}
+(\ell+1)t^1_\ell\frac{\partial}{\partial t^1_\ell}\nonumber\\
&\qquad+\sum_{j=1}^g\left(
(\ell+1)s^j_\ell\frac{\partial}{\partial s^j_\ell}
+\ell\bar s^j_\ell\frac{\partial}{\partial \bar s^j_\ell}
\right)
+\sum_{i=1}^s\sum_{a=1}^{r_i-1}
(\ell+\theta_{i,a})x^{i,a}_\ell\frac{\partial}{\partial x^{i,a}_\ell}
\Bigg)\nonumber\\
&+\chi_{\log}(\cC,Q)\sum_{\ell\ge0}
 t^0_{\ell+1}\frac{\partial}{\partial t^1_\ell}
+\frac{\chi_{\log}(\cC,Q)}{2}(t^0_0)^2
+c_0(\mathcal C/Q),
\end{align*}
where the constant
\begin{equation*}
c_0(\mathcal C/Q)
= \frac14\sum_{i=1}^s\sum_{a=1}^{r_i-1}
  \left(\frac14-\left(\frac{a}{r_i}-\frac12\right)^2\right)=\sum_{i=1}^s\frac{r_i^2-1}{24r_i}.
\end{equation*}
To write $L_n$, $n\geq 1$, we use the Pochhammer symbol:
\begin{equation*}
    (1+x)_n\coloneqq
    \begin{cases}
    (x+1) ... (x+n),\, n>0, \\
    1,\, n=0,\\
    (x(x-1) ... (x+n+1))^{-1},\, n<0.
    \end{cases}
\end{equation*}
And we regard 
\[(\ell)_{n+1}\sum_{j=\ell}^{n+\ell}\frac{1}{j}=\frac{d}{d\ell}(\ell)_{n+1}.\]
In particular, when $\ell=0$, 
\[(\ell)_{n+1}\sum_{j=\ell}^{n+\ell}\frac{1}{j}=n!.\]

For $n\ge1$, define
\begin{align*}
L_n={}&
-(n+1)!\frac{\partial}{\partial t^0_{n+1}}
-\chi_{\log}(\cC, Q)(n+1)!\sum_{j=1}^{n+1}\frac{1}{j}
\frac{\partial}{\partial t^1_n}\nonumber\\
&+\sum_{\ell\ge0}\left[
(\ell)_{n+1}t^0_\ell\frac{\partial}{\partial t^0_{\ell+n}}
+(\ell+1)_{n+1}t^1_\ell\frac{\partial}{\partial t^1_{\ell+n}}
\right]\nonumber\\
&+\sum_{\ell\ge0}\sum_{j=1}^g\left[
(\ell+1)_{n+1}s^j_\ell\frac{\partial}{\partial s^j_{\ell+n}}
+(\ell)_{n+1}\bar s^j_\ell\frac{\partial}{\partial \bar s^j_{\ell+n}}
\right]\nonumber\\
&+\sum_{\ell\ge0}\sum_{i=1}^s\sum_{a=1}^{r_i-1}
(\ell+\theta_{i,a})_{n+1}
 x^{i,a}_\ell\frac{\partial}{\partial x^{i,a}_{\ell+n}}\nonumber\\
&+\chi_{\log}(\cC, Q)\sum_{\ell\ge0}
(\ell)_{n+1}\sum_{j=\ell}^{n+\ell}\frac{1}{j}\,t^0_\ell
\frac{\partial}{\partial t^1_{\ell+n-1}}\nonumber\\
&+\frac{\chi_{\log}(\cC, Q)}2
\sum_{p=0}^{n-2}(p+1)!(n-p-1)!
\frac{\partial^2}{\partial t^1_p\,\partial t^1_{n-p-2}}\nonumber\\
&+\frac{1}{2}\sum_{i=1}^s\sum_{a=1}^{r_i-1}
\sum_{p=0}^{n-1}
r_i(\theta_{i,a})_{p+1}(1-\theta_{i,a})_{n-p}
\frac{\partial^2}{\partial x^{i,a}_p\,\partial x^{i,a^\vee}_{n-p-1}}.
\end{align*}

These operators satisfy
\begin{equation*}
[L_k,L_l]=(k-l)L_{k+l}.
\end{equation*}
When there are no orbifold points (equivalently, $s=0$), these $L_n$ 
specialize to the Okounkov-Pandharipande relative Virasoro operators \cite{OP3}. When $Q$ is empty, they are the specialization of the Jiang–Tseng Virasoro operators for the orbifold curve $\cC$.

\subsection{Main theorem and proof strategy}

\begin{thm}\label{thm:main}
For all $k\ge -1$,
\begin{equation*}
L_kZ^{\cC/Q}_d[\eta^1,\ldots,\eta^m]=0.
\end{equation*}
\end{thm}
We obtain the Virasoro constraints for the absolute orbifold theory of $\cC$ if $m=0$.

\begin{cor}
Jiang-Tseng's Virasoro constraints, Conjecture 3.1 of \cite{JT}, hold for arbitrary smooth projective effective orbifold curves.
\end{cor}

\paragraph{{\bf Proof strategy.}}
The proof consists of a common global-to-local reduction followed by two independent proofs of the local statement.  We first show in Section~\ref{sec:degeneration} that the Virasoro constraints are compatible with the degeneration formula (Proposition \ref{prop:virasoro-defect-degeneration}). Successively separating off the orbifold points by ordinary-node degeneration then reduces Theorem~\ref{thm:main} to the Virasoro constraints for the orbifold caps $(\PP^1_{r,1},\infty)$ where $\PP^1_{r,1}=\PP^1[\sqrt[r]{0}]$ , thanks to the Virasoro constraints for ordinary curves established in \cite{OP3}. We then give two independent proofs for orbifold caps in Part~\ref{part:operator-proof} and 
Part~\ref{part:cap-inversion-proof} respectively. 

The main new ingredient in the first cap proof is an explicit non-equivariant infinite-wedge formula for orbifold caps.  Starting from Johnson's equivariant operator formula for orbifold lines \(\PP^1_{r,s}\) \cite{J}, we combine the large-root operator analysis of Urundolil Kumaran--Wu \cite{KW23} with the orbifold--relative correspondence \cites{TY,FWY,FWY2}, and then extract the large-$s$ constant term and take the non-equivariant limit.  The resulting formula simultaneously encodes arbitrary twisted and untwisted descendants and an arbitrary relative profile.  Consequently, the cap theory need not be reconstructed inductively.  After distinguishing one unit insertion, explicit identities in the operator algebra reproduce, term by term, all linear and quadratic contributions to the Virasoro operators, including the age-dependent twisted-sector terms.  The distinguished insertion is used only to extract this operator identity; no induction on the number of markings or on their descendant degrees is involved.

The second cap proof is independent of the operator formalism. We start from the orbifold lines \(\PP^{1}_{r,1}\). Note that one can degenerate \(\PP^{1}_{r,1}\) to a union of the orbifold cap and the ordinary cap. By the degeneration formula, we obtain a linear system for orbifold cap invariants, with coefficients given by ordinary cap, and constant terms given by \(\PP^{1}_{r,1}\). The full-column-rank property \cite{OP3} of the matrix of ordinary cap invariants allows this system to be solvable uniquely. Observe that the ordinary cap is known to satisfy the relative Virasoro constraint; moreover, the orbifold lines \(\PP^{1}_{r,1}\) have semisimple quantum cohomology \cite{MT1}, and hence the Givental--Teleman reconstruction yields their absolute Virasoro constraints. This allows us to obtain the Virasoro constraints for the orbifold cap. Either cap proof, combined with the same degeneration reduction, proves Theorem~\ref{thm:main}.


\begin{rmk}
Theorem~\ref{thm:main} concerns orbifold curves with trivial generic
stabilizer.  Part~\ref{part:gerbe-extension} considers a smooth proper
one-dimensional Deligne--Mumford stack whose generic stabilizer may be
nontrivial.  The known gerbe-decomposition results reduce the essentially
trivial case, in the sense of \cite{FMN}, to the effective theory proved
here.  For a general gerbe with
trivial band, the same reduction produces Gromov--Witten theories twisted by
flat \(\C^*\)-gerbes; proving the corresponding twisted Virasoro constraints
is the additional step required to extend the theorem to that case. We do
not prove that twisted extension in this paper.
\end{rmk}

\subsection{Organization of the paper}
Section~\ref{sec:degeneration} establishes the compatibility of Virasoro constraints with degeneration formula
and the reduction to orbifold caps.  Part~\ref{part:operator-proof} develops
the infinite-wedge operator formula and gives the first cap proof.
Part~\ref{part:cap-inversion-proof} proves the cap inversion statement and
gives the second cap proof.  Part~\ref{part:gerbe-extension} explains the
extension to essentially trivial gerbes over effective orbifold curves. Appendix~\ref{app:vacuum} contains the auxiliary commutator and vacuum
calculations used in the first proof.

\subsection{Acknowledgment}
Hsian-Hua Tseng thanks Paul Johnson for early collaborations. Part of this work was carried out during Longting Wu’s visit to ETH Z\"urich. He thanks Rahul Pandharipande for his hospitality. Pengrong Jin and Longting Wu are supported by the National Key R\&D Program of China (No. 2022YFA1006200) and NSFC 12522106. Jianxun Hu and Hua-Zhong Ke are supported by the National Key R\&D Program of China (No. 2023YFA1009801) and NSFC (No. 12531003 for J.Hu, No. 12271532 for H.-Z.Ke).

\section{Degeneration and reduction}\label{sec:degeneration}

This section proves the global-to-local reduction from orbifold curves to orbifold caps $\OrbCap_r:=(\mathbb P^1_{r,1},\infty)$ (Proposition \ref{prop:red2orbicap}).

In this section, let $(\mathcal{C},Q)$ be a curve pair as in
Section~\ref{section-relativeorbifoldGW}, and consider a degeneration
\[
(\cC,Q)
\rightsquigarrow
(\cC',Q'\cup\{p\})
\cup_p
(\cC_0,Q_0\cup\{p\}),
\qquad
Q=Q'\sqcup Q_0,
\]
where \(p\) is an ordinary non-orbifold point disjoint from \(Q\), and the
genus of \(\cC_0\) is zero.  Write
\[
\boldsymbol\eta=(\eta^1,\ldots,\eta^m)
\]
for the relative profiles along \(Q\), and let \(\boldsymbol\eta'\) and
\(\boldsymbol\eta_0\) be the corresponding relative profiles along \(Q'\) and \(Q_0\), respectively.
For this degeneration, we choose representatives
\((\gamma',\gamma_0)\in H^*_{\CR}(\cC')\times H^*_{\CR}(\cC_0)\) of
\(\gamma\in H^*_{\CR}(\cC)\) as follows.  For \(\gamma=\mathbf 1\), we choose
\((\gamma',\gamma_0)=(\mathbf 1,\mathbf 1)\).  For \(\gamma=\alpha\), we choose
\((\gamma',\gamma_0)=(\alpha,0)\), and we make the same choice for
\(\gamma=\beta\).  For \(\gamma=\phi_{i,a}\), we choose
\((\gamma',\gamma_0)=(\phi_{i,a},0)\) when \(p_i\in\cC'\), and
\((\gamma',\gamma_0)=(0,\phi_{i,a})\) when \(p_i\in\cC_0\).  We specify
particular representatives for \(\gamma=\omega\) when needed.


\subsection{The relative-orbifold degeneration formula}

The degeneration formula of Gromov-Witten invariants was established by Li--Ruan \cite{LR}, Ionel--Parker \cites{IPsum} and Li \cite{Jun2} for manifolds, and by  Chen--Li--Sun--Zhao
\cite{CLSZ} and
Abramovich--Fantechi \cite{AF} for orbifolds.  

We record the degeneration formula which will be used in this article. Recall that on our source curves, absolute markings are labeled, whereas
relative profiles are unlabelled. We use the Abramovich--Fantechi formalism.

\begin{lem}
\label{lem:relative-orbifold-degeneration}
With the notation fixed above,
\begin{equation}
\label{eq:relative-orbifold-degeneration}
\begin{aligned}
\left\langle
\prod_{i\in N}\tau_{k_i}(\gamma_i)
\ \middle|\
\boldsymbol\eta',\boldsymbol\eta_0
\right\rangle^{\bullet,\,\cC/Q}_{d}
={}&
\sum_{\mu\vdash d}\mathfrak z(\mu)
\sum_{N=N'\sqcup N_0}(-1)^{\epsilon(N',N_0)}
\\[-1mm]
&\hspace{-35mm}\cdot
\left\langle
\prod_{i\in N'}\tau_{k_i}(\gamma_i)
\ \middle|\
\boldsymbol\eta',\mu
\right\rangle^{\bullet,\,\cC'/(Q'\cup\{p\})}_{d}
\left\langle
\prod_{i\in N_0}\tau_{k_i}(\gamma_{i})
\ \middle|\
\boldsymbol\eta_0,\mu
\right\rangle^{\bullet,\,\cC_0/(Q_0\cup\{p\})}_{d}.
\end{aligned}
\end{equation}
Here \(
\mathfrak z(\mu)
=
|\Aut(\mu)|
\prod_{j=1}^{\ell(\mu)}\mu_j
\). For each $\gamma_i\in H^*_{\CR}(\cC)$ we fix its specializations in $H^*_{\CR}(\cC')$ and $H^*_{\CR}(\cC_0)$. The subsets \(N'\) and \(N_0\) are the sets of the absolute markings
specialized to \(\cC'\) and \(\cC_0\), respectively, and the sign \((-1)^{\epsilon(N',N_0)}\) is fixed in terms of the parity of insertions so that we formally have \[\prod_{i\in N}\gamma_i=(-1)^{\epsilon(N',N_0)}\prod_{i\in N'}\gamma_i\prod_{i\in N_0}\gamma_i,\]
\end{lem}
\begin{rmk}
When $\cC$ is an ordinary curve and all $\gamma_i$'s have topologically even degree, \eqref{eq:relative-orbifold-degeneration} is a special case of equation (1.4) of \cite{OP3}.
\end{rmk}

\begin{proof}

When \(Q=\varnothing\), the proof of
\eqref{eq:relative-orbifold-degeneration} is an application of the orbifold
degeneration formula \cite[Theorem~5.9]{AF}. Note that in our degeneration, the
gluing divisor is the non-orbifold point \(p\). It follows that the sum
over the cohomology of the inertia stack of the gluing divisor is
trivial, and the relative conditions on the two branches are indexed
by the same partition \(\mu\). Passing from labeled relative profiles in the Abramovich--Fantechi
formula to unlabelled relative profiles gives the factor
\(\mathfrak z(\mu)\).

When $Q\ne\varnothing$, the proof of
\eqref{eq:relative-orbifold-degeneration} follows by the corresponding analogue of \cite[Theorem~5.9]{AF}. To prove this analogue result, note that the points of \(Q\) are disjoint from $p$. Expansions along
\(Q\) and expansions at $p$ can therefore be formed
independently.  Each relative point of \(Q\), together with its relative
markings and contact orders, remains on the unique component containing
that point. Consequently, the normalization and gluing construction at \(p\), as
well as the comparison of the obstruction theories, is unchanged from
the proof of \cite[Theorem~5.9]{AF}.  The relative conditions
\(\boldsymbol\eta'\) and \(\boldsymbol\eta_0\) are carried as fixed
additional data and introduce neither new matching conditions nor
additional gluing factors. We leave details of verification to interested readers.
\end{proof}

\subsection{Reduction to the orbifold cap}


Let $\mathbf B$ be an ordered absolute insertion list with insertions of type $\tau_\ell(\gamma)$, where $\gamma\in\{\id,\omega,\alpha_j,\beta_j,\phi_{i,a}\}$.  For
\(A=\tau_\ell(\gamma)\), when \(\gamma\) has even ordinary cohomological
degree, we say that \(A\) is even, and let \(m_A(\mathbf B)\)
be its multiplicity in \(\mathbf B\). Set
\[
\mathbf B!
:=
\prod_{A\ {\rm even}}m_A(\mathbf B)!.
\]
Let $\mathbf t_A$ be the super-variable for \(A=\tau_\ell(\gamma)\) whose parity depends on $\gamma$, i.e.
\[
\mathbf t_A=
\begin{cases}
t^0_\ell,&A=\tau_\ell(\id),\\
t^1_\ell,&A=\tau_\ell(\omega),\\
s^j_\ell,&A=\tau_\ell(\alpha_j),\quad 1\leq j\leq g,\\
\bar s^j_\ell,&A=\tau_\ell(\beta_j),\quad 1\leq j\leq g,\\
x^{i,a}_\ell,&A=\tau_\ell(\phi_{i,a}),\quad 1\leq i\leq s,\ 1\leq a<r_i.
\end{cases}
\]
For $\mathbf B=[A_1,\dots,A_m]$, set $\mathbf t_\mathbf B:=\mathbf t_{A_1}\cdots\mathbf t_{A_m}$.

For the relative data \(\boldsymbol\eta\) along \(Q\), define the coefficientwise
Virasoro defect of \(\mathcal{C}/Q\) by
\[
\Delta_{k,d}^{\mathcal{C}/Q}
\left(
\mathbf B
\ \middle|\
\boldsymbol\eta
\right)
\coloneqq
\mathbf B!\,
\left[\mathbf t_{\mathbf B}\right]
\left(L_kZ_d^{\mathcal{C}/Q}[\boldsymbol\eta]\right).
\]
Here \([\mathbf t_{\mathbf B}](\dots)\) denotes the coefficient of $\mathbf t_{\mathbf B}$ in \((\dots)\).
The Virasoro conjecture for \(\mathcal{C}/Q\) is therefore equivalent to
\begin{align}\label{formula-virasorodefect}
\Delta_{k,d}^{\mathcal{C}/Q}
\left(\mathbf B\ \middle|\ \boldsymbol\eta\right)=0
\end{align}
for any \(k\geq-1\), \(d\geq0\), any ordered absolute insertion lists
\(\mathbf B\), and any relative data \(\boldsymbol\eta\).

\begin{prop}
\label{prop:virasoro-defect-degeneration}
Use the relative-data notation fixed at the beginning of this section.
Assume that the coarse curve underlying \(\cC_0\)
has genus zero.
Let
\[
\mathbf B=
\prod_{i=1}^{u}\tau_{a_i}(\id)
\prod_{j=1}^{v}\tau_{b_j}(\omega)
\mathbf A'\mathbf A_0,
\]
where \(\mathbf A'\) and \(\mathbf A_0\) are the ordered products of the
remaining insertions supported on \(\cC'\) and \(\cC_0\), respectively. We fix a partition \(\{1,\ldots,v\}=T\sqcup T'\), such that for $j\in T$, the representative of $\omega$ in $\tau_{b_j}(\omega)$ is chosen to be $(\omega',\omega_0)=(0,\omega)$, while for $j\in T'$, the representative of $\omega$ in $\tau_{b_j}(\omega)$ is chosen to be $(\omega',\omega_0)=(\omega,0)$. For \(S\subset\{1,\ldots,u\}\), set
\[
\mathbf B'(S)
=
\prod_{i\notin S}\tau_{a_i}(\id)
\prod_{j\notin T}\tau_{b_j}(\omega)
\mathbf A',
\qquad
\mathbf B_0(S)
=
\prod_{i\in S}\tau_{a_i}(\id)
\prod_{j\in T}\tau_{b_j}(\omega)
\mathbf A_0.
\]

Then, for every \(k\geq-1\),
\begin{equation}
\label{eq:virasoro-defect-degeneration}
\begin{aligned}
&
\Delta_{k,d}^{\cC/Q}
\left(
\mathbf B
\ \middle|\
\boldsymbol\eta
\right)
\\
&=
\sum_{S\subset\{1,\ldots,u\}}
\sum_{\mu\vdash d}
\mathfrak z(\mu)
\Bigg[
\Delta_{k,d}^{\cC'/(Q'\cup\{p\})}
\left(
\mathbf B'(S)
\ \middle|\
\boldsymbol\eta',\mu
\right)\cdot
\left\langle
\mathbf B_0(S)
\ \middle|\
\boldsymbol\eta_0,\mu
\right\rangle^{\bullet,\,\cC_0/(Q_0\cup\{p\})}_{d}
\\
&\hspace{25mm}
+
\left\langle
\mathbf B'(S)
\ \middle|\
\boldsymbol\eta',\mu
\right\rangle^{\bullet,\,\cC'/(Q'\cup\{p\})}_{d}\cdot
\Delta_{k,d}^{\cC_0/(Q_0\cup\{p\})}
\left(
\mathbf B_0(S)
\ \middle|\
\boldsymbol\eta_0,\mu
\right)
\Bigg].
\end{aligned}
\end{equation}
\end{prop}

\begin{proof}

We use the following notations in the proof. Recall that an insertion type is of the form \(\tau_\ell(\gamma)\).
Let \(A\) and \(B\) be insertion types.
For an ordered insertion list \(\mathbf D\), let
\(\Lambda_A(\mathbf D)\) be the set of markings with insertion of type
\(A\) in \(\mathbf D\).  For
\(\lambda\in\Lambda_A(\mathbf D)\), let
\(\mathbf D_{\lambda:A\rightsquigarrow B}\) be the new ordered insertion list obtained by replacing the type \(A\) at the
marking \(\lambda\) in $\mathbf D$ by \(B\).  Let
\(\mathbf D\sqcup\{A\}\) be the  new ordered insertion list obtained by adding one new insertion of type \(A\) after $\mathbf D$.
For two distinct markings \(\lambda,\nu\) of \(\mathbf D\), let
\(\mathbf D\setminus\{\lambda,\nu\}\) be the new  ordered insertion list obtained by deleting the corresponding insertions. 

With these conventions, all four identities below hold when \(A\) and \(B\) are
both even.  The second identity also holds when \(A\) and \(B\) are
both odd:
\begin{align*}
\mathbf D!\,[\mathbf t_{\mathbf D}]\frac{\partial}{\partial\mathbf t_A} Z_d^{\cC/Q}[\boldsymbol\eta]&=\left\langle\mathbf D\sqcup\{A\}\ \middle|\ \boldsymbol\eta\right\rangle_d^{\bullet,\,\cC/Q},\\
\mathbf D!\,[\mathbf t_{\mathbf D}]\,\mathbf t_A\frac{\partial}{\partial\mathbf t_B} Z_d^{\cC/Q}[\boldsymbol\eta]&=\sum_{\lambda\in\Lambda_A(\mathbf D)}\left\langle\mathbf D_{\lambda:A\rightsquigarrow B}\ \middle|\ \boldsymbol\eta\right\rangle_d^{\bullet,\,\cC/Q},\\
\mathbf D!\,[\mathbf t_{\mathbf D}]\frac{\partial^2 }{\partial\mathbf t_A\,\partial\mathbf t_B}Z_d^{\cC/Q}[\boldsymbol\eta]&=\left\langle\mathbf D\sqcup\{A,B\}\ \middle|\ \boldsymbol\eta\right\rangle_d^{\bullet,\,\cC/Q},\\
\mathbf D!\,[\mathbf t_{\mathbf D}]\,\mathbf t_A\mathbf t_B Z_d^{\cC/Q}[\boldsymbol\eta]&=\sum_{\substack{\lambda\in\Lambda_A(\mathbf D),\,\nu\in\Lambda_B(\mathbf D)\\ \lambda\ne\nu}}\left\langle\mathbf D\setminus\{\lambda,\nu\}\ \middle|\ \boldsymbol\eta\right\rangle_d^{\bullet,\,\cC/Q}.
\end{align*}

In what follows, we suppress target-independent scalar coefficients of the operators for ease of notation, but retain
the coefficients involving \(\chi_{\log}\) which is target-dependent. Note that
\begin{equation}\label{eq:chi-log-degeneration-in-proof}
\chi_{\log}(\cC,Q)=\chi_{\log}(\cC',Q'\cup\{p\})+\chi_{\log}(\cC_0,Q_0\cup\{p\}).
\end{equation}

For \(\partial/\partial t^0_{k+1}\), we add
\(\tau_{k+1}(\id)\) in the insertions. Now \eqref{eq:relative-orbifold-degeneration} gives
\[
\begin{aligned}
&\left\langle\mathbf B\sqcup\{\tau_{k+1}(\id)\}
 \ \middle|\ \boldsymbol\eta',\boldsymbol\eta_0\right\rangle_d^{\bullet,\,\cC/Q}
={}\sum_{S\subset\{1,\ldots,u\}}\sum_{\mu\vdash d}\mathfrak z(\mu)\Bigg[
\\[-2mm]
&\quad\left\langle\mathbf B'(S)\sqcup\{\tau_{k+1}(\id)\}
 \ \middle|\ \boldsymbol\eta',\mu\right\rangle_d^{\bullet,\,\cC'/(Q'\cup\{p\})}
 \left\langle\mathbf B_0(S)\ \middle|\ \boldsymbol\eta_0,\mu\right\rangle_d^{\bullet,\,\cC_0/(Q_0\cup\{p\})}
\\[-1mm]
&\quad+\left\langle\mathbf B'(S)\ \middle|\ \boldsymbol\eta',\mu\right\rangle_d^{\bullet,\,\cC'/(Q'\cup\{p\})}
 \left\langle\mathbf B_0(S)\sqcup\{\tau_{k+1}(\id)\}
 \ \middle|\ \boldsymbol\eta_0,\mu\right\rangle_d^{\bullet,\,\cC_0/(Q_0\cup\{p\})}\Bigg].
\end{aligned}
\]

For \(-\chi_{\log}(\cC,Q)\,\partial/\partial t^1_k\), we add
\(\tau_k(\omega)\) in the insertions. From \eqref{eq:chi-log-degeneration-in-proof}, we have
\begin{align*}
 &-\chi_{\log}(\cC,Q)
 \left\langle\mathbf B\sqcup\{\tau_k(\omega)\}
 \ \middle|\ \boldsymbol\eta',\boldsymbol\eta_0\right\rangle_d^{\bullet,\,\cC/Q}\\
 &=
 -\chi_{\log}(\cC',Q'\cup\{p\})
 \left\langle\mathbf B\sqcup\{\tau_k(\omega)\}
 \ \middle|\ \boldsymbol\eta',\boldsymbol\eta_0\right\rangle_d^{\bullet,\,\cC/Q}\\
 &\hspace*{14mm}-\chi_{\log}(\cC_0,Q_0\cup\{p\})
 \left\langle\mathbf B\sqcup\{\tau_k(\omega)\}
 \ \middle|\ \boldsymbol\eta',\boldsymbol\eta_0\right\rangle_d^{\bullet,\,\cC/Q}\\
 &=(I)+(II).
\end{align*}
For the added $\tau_k(\omega)$ in $(I)$, we choose the representative $(\omega',\omega_0)=(\omega,0)$, and for the added $\tau_k(\omega)$ in $(II)$, we choose the representative $(\omega',\omega_0)=(0,\omega)$. Now \eqref{eq:relative-orbifold-degeneration} gives
\[
\begin{aligned}
&-\chi_{\log}(\cC,Q)
 \left\langle\mathbf B\sqcup\{\tau_k(\omega)\}
 \ \middle|\ \boldsymbol\eta',\boldsymbol\eta_0\right\rangle_d^{\bullet,\,\cC/Q}
={}\sum_{S\subset\{1,\ldots,u\}}\sum_{\mu\vdash d}\mathfrak z(\mu)\Bigg[
\\[-2mm]
&\quad-\chi_{\log}(\cC',Q'\cup\{p\})
 \left\langle\mathbf B'(S)\sqcup\{\tau_k(\omega)\}
 \,\middle|\,\boldsymbol\eta',\mu\right\rangle_d^{\bullet,\,\cC'/(Q'\cup\{p\})}
 \left\langle\mathbf B_0(S)\,\middle|\,\boldsymbol\eta_0,\mu\right\rangle_d^{\bullet,\,\cC_0/(Q_0\cup\{p\})}
\\[-1mm]
&\quad-\chi_{\log}(\cC_0,Q_0\cup\{p\})
 \left\langle\mathbf B'(S)\,\middle|\,\boldsymbol\eta',\mu\right\rangle_d^{\bullet,\,\cC'/(Q'\cup\{p\})}
 \left\langle\mathbf B_0(S)\sqcup\{\tau_k(\omega)\}
 \,\middle|\,\boldsymbol\eta_0,\mu\right\rangle_d^{\bullet,\,\cC_0/(Q_0\cup\{p\})}\Bigg].
\end{aligned}
\]

For \(t_A\partial/\partial t_B\), where \(A=\tau_\ell(\gamma)\), \(B=\tau_{\ell+k}(\gamma)\) and
\(\gamma\in\{\id,\omega,\alpha_j,\beta_j,\phi_{i,a}\}\), we replace
\(A\) in \(\mathbf B\) by \(B\).  Now \eqref{eq:relative-orbifold-degeneration} gives
\[
\begin{aligned}
&\sum_{\lambda\in\Lambda_A(\mathbf B)}
 \left\langle\mathbf B_{\lambda:A\rightsquigarrow B}
 \ \middle|\ \boldsymbol\eta',\boldsymbol\eta_0\right\rangle_d^{\bullet,\,\cC/Q}
={}\sum_{S\subset\{1,\ldots,u\}}\sum_{\mu\vdash d}\mathfrak z(\mu)\Bigg[
\\[-2mm]
&\quad\sum_{\lambda\in\Lambda_A(\mathbf B'(S))}
 \left\langle\mathbf B'(S)_{\lambda:A\rightsquigarrow B}
 \ \middle|\ \boldsymbol\eta',\mu\right\rangle_d^{\bullet,\,\cC'/(Q'\cup\{p\})}
 \left\langle\mathbf B_0(S)\ \middle|\ \boldsymbol\eta_0,\mu\right\rangle_d^{\bullet,\,\cC_0/(Q_0\cup\{p\})}
\\[-1mm]
&\quad+\left\langle\mathbf B'(S)\ \middle|\ \boldsymbol\eta',\mu\right\rangle_d^{\bullet,\,\cC'/(Q'\cup\{p\})}
 \sum_{\lambda\in\Lambda_A(\mathbf B_0(S))}
 \left\langle\mathbf B_0(S)_{\lambda:A\rightsquigarrow B}
 \ \middle|\ \boldsymbol\eta_0,\mu\right\rangle_d^{\bullet,\,\cC_0/(Q_0\cup\{p\})}\Bigg].
\end{aligned}
\]

For \(\chi_{\log}(\cC,Q)t^0_\ell\partial_{t^1_{\ell+k-1}}\), setting
\(A=\tau_\ell(\id)\), \(B=\tau_{\ell+k-1}(\omega)\), we replace \(A\) in $\mathbf B$ by \(B\).
Following similar approach for \(-\chi_{\log}(\cC,Q)\,\partial/\partial t^1_k\), we get
\[
\begin{aligned}
&\chi_{\log}(\cC,Q)\sum_{\lambda\in\Lambda_A(\mathbf B)}
 \left\langle\mathbf B_{\lambda:A\rightsquigarrow B}
 \ \middle|\ \boldsymbol\eta',\boldsymbol\eta_0\right\rangle_d^{\bullet,\,\cC/Q}
={}\sum_{S\subset\{1,\ldots,u\}}\sum_{\mu\vdash d}\mathfrak z(\mu)\Bigg[
\\[-2mm]
&\quad\chi_{\log}(\cC',Q'\cup\{p\})
 \sum_{\lambda\in\Lambda_A(\mathbf B'(S))}
 \left\langle\mathbf B'(S)_{\lambda:A\rightsquigarrow B}
 \ \middle|\ \boldsymbol\eta',\mu\right\rangle_d^{\bullet,\,\cC'/(Q'\cup\{p\})}
 \left\langle\mathbf B_0(S)\ \middle|\ \boldsymbol\eta_0,\mu\right\rangle_d^{\bullet,\,\cC_0/(Q_0\cup\{p\})}
\\[-2mm]
&+\chi_{\log}(\cC_0,Q_0\cup\{p\})
 \sum_{\lambda\in\Lambda_A(\mathbf B_0(S))}
 \left\langle\mathbf B'(S)\ \middle|\ \boldsymbol\eta',\mu\right\rangle_d^{\bullet,\,\cC'/(Q'\cup\{p\})}
 \left\langle\mathbf B_0(S)_{\lambda:A\rightsquigarrow B}
 \ \middle|\ \boldsymbol\eta_0,\mu\right\rangle_d^{\bullet,\,\cC_0/(Q_0\cup\{p\})}\Bigg].
\end{aligned}
\]

For
\(\chi_{\log}(\cC,Q)\partial^2/(\partial t^1_q\,\partial t^1_{k-q-2})\),
setting \(A=\tau_q(\omega)\) and \(B=\tau_{k-q-2}(\omega)\), we add \(\{A,B\}\) in the insertions.
Following similar approach for \(-\chi_{\log}(\cC,Q)\,\partial/\partial t^1_k\), we get
\[
\begin{aligned}
&\chi_{\log}(\cC,Q)
 \left\langle\mathbf B\sqcup\{A,B\}
 \ \middle|\ \boldsymbol\eta',\boldsymbol\eta_0\right\rangle_d^{\bullet,\,\cC/Q}
={}\sum_{S\subset\{1,\ldots,u\}}\sum_{\mu\vdash d}\mathfrak z(\mu)\Bigg[
\\[-2mm]
&\quad\chi_{\log}(\cC',Q'\cup\{p\})
 \left\langle\mathbf B'(S)\sqcup\{A,B\}
 \,\middle|\,\boldsymbol\eta',\mu\right\rangle_d^{\bullet,\,\cC'/(Q'\cup\{p\})}
 \left\langle\mathbf B_0(S)\,\middle|\,\boldsymbol\eta_0,\mu\right\rangle_d^{\bullet,\,\cC_0/(Q_0\cup\{p\})}
\\[-1mm]
&\quad+\chi_{\log}(\cC_0,Q_0\cup\{p\})
 \left\langle\mathbf B'(S)\,\middle|\,\boldsymbol\eta',\mu\right\rangle_d^{\bullet,\,\cC'/(Q'\cup\{p\})}
 \left\langle\mathbf B_0(S)\sqcup\{A,B\}
 \,\middle|\,\boldsymbol\eta_0,\mu\right\rangle_d^{\bullet,\,\cC_0/(Q_0\cup\{p\})}\Bigg].
\end{aligned}
\]

For
\(\partial^2/(\partial x^{i,a}_q\,\partial x^{i,a^\vee}_{k-q-1})\), setting
\(A=\tau_q(\phi_{i,a})\) and
\(B=\tau_{k-q-1}(\phi_{i,a^\vee})\),  we add \(\{A,B\}\) in the insertions. Applying \eqref{eq:relative-orbifold-degeneration}, if \(p_i\in\cC'\), then
\[
\begin{aligned}
&\left\langle\mathbf B\sqcup\{A,B\}
 \ \middle|\ \boldsymbol\eta',\boldsymbol\eta_0\right\rangle_d^{\bullet,\,\cC/Q}
\\
={}&\sum_{S\subset\{1,\ldots,u\}}\sum_{\mu\vdash d}\mathfrak z(\mu)
 \left\langle\mathbf B'(S)\sqcup\{A,B\}
 \ \middle|\ \boldsymbol\eta',\mu\right\rangle_d^{\bullet,\,\cC'/(Q'\cup\{p\})}
 \left\langle\mathbf B_0(S)\ \middle|\ \boldsymbol\eta_0,\mu\right\rangle_d^{\bullet,\,\cC_0/(Q_0\cup\{p\})},
\end{aligned}
\]
and if \(p_i\in\cC_0\), then
\[
\begin{aligned}
&\left\langle\mathbf B\sqcup\{A,B\}
 \ \middle|\ \boldsymbol\eta',\boldsymbol\eta_0\right\rangle_d^{\bullet,\,\cC/Q}
\\
={}&\sum_{S\subset\{1,\ldots,u\}}\sum_{\mu\vdash d}\mathfrak z(\mu)
 \left\langle\mathbf B'(S)
 \ \middle|\ \boldsymbol\eta',\mu\right\rangle_d^{\bullet,\,\cC'/(Q'\cup\{p\})}
 \left\langle\mathbf B_0(S)\sqcup\{A,B\}\ \middle|\ \boldsymbol\eta_0,\mu\right\rangle_d^{\bullet,\,\cC_0/(Q_0\cup\{p\})};
\end{aligned}
\]

For \(\chi_{\log}(\cC,Q)(t^0_0)^2\), setting \(A=\tau_0(\id)\), we delete two distinct markings of \(\mathbf B\) with insertion
\(A\).  From \eqref{eq:chi-log-degeneration-in-proof}, we have
\begin{align*}
    &\chi_{\log}(\cC,Q)
 \sum_{\substack{\lambda,\nu\in\Lambda_A(\mathbf B)\\ \lambda\ne\nu}}
 \left\langle\mathbf B\setminus\{\lambda,\nu\}
 \ \middle|\ \boldsymbol\eta',\boldsymbol\eta_0\right\rangle_d^{\bullet,\,\cC/Q}\\
 ={}&\chi_{\log}(\cC',Q'\cup\{p\})
 \sum_{\substack{\lambda,\nu\in\Lambda_A(\mathbf B)\\ \lambda\ne\nu}}
 \left\langle\mathbf B\setminus\{\lambda,\nu\}
 \ \middle|\ \boldsymbol\eta',\boldsymbol\eta_0\right\rangle_d^{\bullet,\,\cC/Q}\\
 &\qquad+\chi_{\log}(\cC_0,Q_0\cup\{p\})
 \sum_{\substack{\lambda,\nu\in\Lambda_A(\mathbf B)\\ \lambda\ne\nu}}
 \left\langle\mathbf B\setminus\{\lambda,\nu\}
 \ \middle|\ \boldsymbol\eta',\boldsymbol\eta_0\right\rangle_d^{\bullet,\,\cC/Q}\\
 ={}&(I)+(II).
\end{align*}
Applying \eqref{eq:relative-orbifold-degeneration}, we get
\begin{align*}
    (I)&=\chi_{\log}(\cC',Q'\cup\{p\})\sum_{\mu\vdash d}\mathfrak z(\mu)\\
    &\cdot\sum_{\substack{S\subset\{1,\ldots,u\}\\\lambda,\nu\in\Lambda_A(\mathbf B'(S)), \lambda\ne\nu}}
 \left\langle\mathbf B'(S)\setminus\{\lambda,\nu\}
 \ \middle|\ \boldsymbol\eta',\mu\right\rangle_d^{\bullet,\,\cC'/(Q'\cup\{p\})}
 \left\langle\mathbf B_0(S)\ \middle|\ \boldsymbol\eta_0,\mu\right\rangle_d^{\bullet,\,\cC_0/(Q_0\cup\{p\})},
\end{align*}
and 
\begin{align*}
    (II)&=\chi_{\log}(\cC_0,Q_0\cup\{p\})\sum_{\mu\vdash d}\mathfrak z(\mu)\\
    &\cdot\sum_{\substack{S\subset\{1,\ldots,u\}\\\lambda,\nu\in\Lambda_A(\mathbf B_0(S)), \lambda\ne\nu}}
 \left\langle\mathbf B'(S)
 \ \middle|\ \boldsymbol\eta',\mu\right\rangle_d^{\bullet,\,\cC'/(Q'\cup\{p\})}
 \left\langle\mathbf B_0(S)\setminus\{\lambda,\nu\}\ \middle|\ \boldsymbol\eta_0,\mu\right\rangle_d^{\bullet,\,\cC_0/(Q_0\cup\{p\})}.
\end{align*}


For $t^0_0t^1_0$, setting $A=\tau_0(\id)$ and $B=\tau_0(\omega)$, we delete two markings $\lambda\in\Lambda_A(\mathbf B)$ and $\nu\in\Lambda_B(\mathbf B)$ in $\mathbf B$. Applying \eqref{eq:relative-orbifold-degeneration}, if $\nu\in T$, then
\[
\begin{aligned}
 &\left\langle\mathbf B\setminus\{\lambda,\nu\}
 \ \middle|\ \boldsymbol\eta',\boldsymbol\eta_0\right\rangle_d^{\bullet,\,\cC/Q}
\\
={}&\sum_{\mu\vdash d}\mathfrak z(\mu)
\sum_{\substack{S\subset\{1,\ldots,u\}\\\lambda\in\Lambda_A(\mathbf B_0(S))\\\nu\in\Lambda_B(\mathbf B_0(S))}}
 \left\langle\mathbf B'(S)\ \middle|\ \boldsymbol\eta',\mu\right\rangle_d^{\bullet,\,\cC'/(Q'\cup\{p\})}
 \left\langle\mathbf B_0(S)\setminus\{\lambda,\nu\}
 \ \middle|\ \boldsymbol\eta_0,\mu\right\rangle_d^{\bullet,\,\cC_0/(Q_0\cup\{p\})},
\end{aligned}
\]
and if $\nu\in T'$, then
\[
\begin{aligned}
&\left\langle\mathbf B\setminus\{\lambda,\nu\}
 \ \middle|\ \boldsymbol\eta',\boldsymbol\eta_0\right\rangle_d^{\bullet,\,\cC/Q}
\\
={}&\sum_{\mu\vdash d}\mathfrak z(\mu)
 \sum_{\substack{S\subset\{1,\ldots,u\}\\\lambda\in\Lambda_A(\mathbf B'(S))\\\nu\in\Lambda_B(\mathbf B'(S))}}
 \left\langle\mathbf B'(S)\setminus\{\lambda,\nu\}
 \ \middle|\ \boldsymbol\eta',\mu\right\rangle_d^{\bullet,\,\cC'/(Q'\cup\{p\})}
 \left\langle\mathbf B_0(S)\ \middle|\ \boldsymbol\eta_0,\mu\right\rangle_d^{\bullet,\,\cC_0/(Q_0\cup\{p\})}.
\end{aligned}
\]
Observe that, when $\nu\in T$, the condition $\nu\in\Lambda_B(\mathbf B'(S))$ is empty, and when $\nu\in T'$, the condition $\nu\in\Lambda_B(\mathbf B_0(S))$ is empty. Consequently, we obtain
\[
\begin{aligned}
&\sum_{\substack{\lambda\in\Lambda_A(\mathbf B)\\\nu\in\Lambda_B(\mathbf B)}}
 \left\langle\mathbf B\setminus\{\lambda,\nu\}
 \ \middle|\ \boldsymbol\eta',\boldsymbol\eta_0\right\rangle_d^{\bullet,\,\cC/Q}
={}\sum_{\substack{\lambda\in\Lambda_A(\mathbf B)\\\nu\in\Lambda_B(\mathbf B)}}\sum_{\mu\vdash d}\mathfrak z(\mu)\Bigg[
\\[-2mm]
&\quad
 \sum_{\substack{S\subset\{1,\ldots,u\}\\\lambda\in\Lambda_A(\mathbf B'(S))\\\nu\in\Lambda_B(\mathbf B'(S))}}
 \left\langle\mathbf B'(S)\setminus\{\lambda,\nu\}
 \ \middle|\ \boldsymbol\eta',\mu\right\rangle_d^{\bullet,\,\cC'/(Q'\cup\{p\})}
 \left\langle\mathbf B_0(S)\ \middle|\ \boldsymbol\eta_0,\mu\right\rangle_d^{\bullet,\,\cC_0/(Q_0\cup\{p\})}
\\[-1mm]
&\quad+\sum_{\substack{S\subset\{1,\ldots,u\}\\\lambda\in\Lambda_A(\mathbf B_0(S))\\\nu\in\Lambda_B(\mathbf B_0(S))}}
 \left\langle\mathbf B'(S)\ \middle|\ \boldsymbol\eta',\mu\right\rangle_d^{\bullet,\,\cC'/(Q'\cup\{p\})}
 \left\langle\mathbf B_0(S)\setminus\{\lambda,\nu\}
 \ \middle|\ \boldsymbol\eta_0,\mu\right\rangle_d^{\bullet,\,\cC_0/(Q_0\cup\{p\})}\Bigg].
\end{aligned}
\]

For \(x^{i,a}_0x^{i,a^\vee}_0\), setting
\(A=\tau_0(\phi_{i,a})\) and \(B=\tau_0(\phi_{i,a^\vee}))\), 
we delete two markings $\lambda\in\Lambda_A(\mathbf B)$ and $\nu\in\Lambda_B(\mathbf B)$ in $\mathbf B$.
Applying \eqref{eq:relative-orbifold-degeneration}, if $p_i\in\cC'$, then
\[
\begin{aligned}
&\sum_{\substack{\lambda\in\Lambda_A(\mathbf B)\\\nu\in\Lambda_B(\mathbf B)}}\left\langle\mathbf B\setminus\{\lambda,\nu\}
 \ \middle|\ \boldsymbol\eta',\boldsymbol\eta_0\right\rangle_d^{\bullet,\,\cC/Q}
={}\sum_{\substack{\lambda\in\Lambda_A(\mathbf B)\\\nu\in\Lambda_B(\mathbf B)}}
 \sum_{\substack{S\subset\{1,\ldots,u\}\\\lambda\in\Lambda_A(\mathbf B'(S))\\\nu\in\Lambda_B(\mathbf B'(S))}}
 \sum_{\mu\vdash d}\mathfrak z(\mu)
\\[-1mm]
&\hspace*{20mm}\cdot
 \left\langle\mathbf B'(S)\setminus\{\lambda,\nu\}
 \ \middle|\ \boldsymbol\eta',\mu\right\rangle_d^{\bullet,\,\cC'/(Q'\cup\{p\})}
 \left\langle\mathbf B_0(S)\ \middle|\ \boldsymbol\eta_0,\mu\right\rangle_d^{\bullet,\,\cC_0/(Q_0\cup\{p\})},
\end{aligned}
\]
and if $p_i\in\cC_0$, then
\[
\begin{aligned}
&\sum_{\substack{\lambda\in\Lambda_A(\mathbf B)\\\nu\in\Lambda_B(\mathbf B)}}\left\langle\mathbf B\setminus\{\lambda,\nu\}
 \ \middle|\ \boldsymbol\eta',\boldsymbol\eta_0\right\rangle_d^{\bullet,\,\cC/Q}
={}\sum_{\substack{\lambda\in\Lambda_A(\mathbf B)\\\nu\in\Lambda_B(\mathbf B)}}
 \sum_{\substack{S\subset\{1,\ldots,u\}\\\lambda\in\Lambda_A(\mathbf B_0(S))\\\nu\in\Lambda_B(\mathbf B_0(S))}}
 \sum_{\mu\vdash d}\mathfrak z(\mu)
\\[-1mm]
&\hspace*{20mm}\cdot
 \left\langle\mathbf B'(S)
 \ \middle|\ \boldsymbol\eta',\mu\right\rangle_d^{\bullet,\,\cC'/(Q'\cup\{p\})}
 \left\langle\mathbf B_0(S)\setminus\{\lambda,\nu\}\ \middle|\ \boldsymbol\eta_0,\mu\right\rangle_d^{\bullet,\,\cC_0/(Q_0\cup\{p\})},
\end{aligned}
\]

The term \(s^j_0\bar s^j_0\) in \(L_{-1}\) is handled directly at the
level of coefficient extraction.  Since all odd
variables occur on \(\cC'\), it follows that \eqref{eq:relative-orbifold-degeneration} gives
\[
\begin{aligned}
&\mathbf B!\,[\mathbf t_{\mathbf B}]\,s^j_0\bar s^j_0
 Z_d^{\cC/Q}[\boldsymbol\eta]
\\
={}&\sum_{S\subset\{1,\ldots,u\}}\sum_{\mu\vdash d}\mathfrak z(\mu)
 \Bigl(\mathbf B'(S)!\,[\mathbf t_{\mathbf B'(S)}]\,s^j_0\bar s^j_0
 Z_d^{\cC'/(Q'\cup\{p\})}[\boldsymbol\eta',\mu]\Bigr)\cdot
 \left\langle\mathbf B_0(S)\ \middle|\ \boldsymbol\eta_0,\mu\right\rangle_d^{\bullet,\,\cC_0/(Q_0\cup\{p\})}.
\end{aligned}
\]
The relative order of the odd variables is unchanged in
\(\mathbf B'(S)\),
hence the sign is the same on both sides.

For $c_0(\cC/Q)$ in $L_0$, note that
\[
c_0(\cC/Q)
=\sum_{i=1}^{s}\frac{r_i^2-1}{24r_i}
=c_0\bigl(\cC'/(Q'\cup\{p\})\bigr)
+c_0\bigl(\cC_0/(Q_0\cup\{p\})\bigr).
\]
Now  \eqref{eq:relative-orbifold-degeneration} gives
\[
\begin{aligned}
&c_0(\cC/Q)
 \left\langle\mathbf B\ \middle|\ \boldsymbol\eta',\boldsymbol\eta_0
 \right\rangle_d^{\bullet,\,\cC/Q}
={}\sum_{S\subset\{1,\ldots,u\}}\sum_{\mu\vdash d}\mathfrak z(\mu)\Bigg[
\\[-2mm]
&\quad c_0\bigl(\cC'/(Q'\cup\{p\})\bigr)
 \left\langle\mathbf B'(S)\ \middle|\ \boldsymbol\eta',\mu\right\rangle_d^{\bullet,\,\cC'/(Q'\cup\{p\})}
 \left\langle\mathbf B_0(S)\ \middle|\ \boldsymbol\eta_0,\mu\right\rangle_d^{\bullet,\,\cC_0/(Q_0\cup\{p\})}
\\[-1mm]
&\quad+c_0\bigl(\cC_0/(Q_0\cup\{p\})\bigr)
 \left\langle\mathbf B'(S)\ \middle|\ \boldsymbol\eta',\mu\right\rangle_d^{\bullet,\,\cC'/(Q'\cup\{p\})}
 \left\langle\mathbf B_0(S)\ \middle|\ \boldsymbol\eta_0,\mu\right\rangle_d^{\bullet,\,\cC_0/(Q_0\cup\{p\})}\Bigg].
\end{aligned}
\]

Finally, combining the identities above for all operators in \(L_k\) gives
\eqref{eq:virasoro-defect-degeneration}. This completes the proof.
\end{proof}


\begin{prop}\label{prop:red2orbicap}
Assume that Theorem \ref{thm:main} holds for  \(\OrbCap_r\) for all \(r\). Then it holds for all pairs $(\cC,Q)$.
\end{prop}

\begin{proof}
It suffices to prove \eqref{formula-virasorodefect}, and we argue by induction on the number of orbifold points on $\cC$.  When there are none, \eqref{formula-virasorodefect} for smooth curves hold by \cite{OP3}.  When there are at least one orbifold points,
we can separate off one orbifold point of order \(r\) by an ordinary-node degeneration
\[
(\cC,Q)
\rightsquigarrow
(\cC',Q\cup\{p\})
\cup_p
\OrbCap_r,
\]
where \(\cC'\) has one fewer orbifold point.  The Virasoro constraints hold
for \(\cC'/(Q\cup\{p\})\) by induction and for \(\OrbCap_r\) by assumption.
Proposition \ref{prop:virasoro-defect-degeneration} then implies \eqref{formula-virasorodefect}. This completes the induction.
\end{proof}

Proposition~\ref{prop:red2orbicap} reduces the proof of
Theorem~\ref{thm:main} to the Virasoro constraints for the orbifold caps
\(\OrbCap_r\), \(r\geq 1\).  It therefore remains only to establish these
cap constraints.  In the next two parts, we give two independent proofs:
the first uses the orbifold-cap operator formula, while the second proceeds
by degeneration from the absolute $\PP^1_{r,1}$ together with cap inversion.
Either proof, combined with Proposition~\ref{prop:red2orbicap}, completes the
proof of Theorem~\ref{thm:main}.

{\large \part{First orbifold cap proof: the operator-formula method}\label{part:operator-proof}}


\section{Infinite wedge formalism}\label{sec:wedge}

This section reviews the infinite wedge formalism used later.  The exposition follows the conventions of \cites{OP3,J,KW23}. We include the definitions needed for the orbifold cap proof.

\subsection{The infinite wedge space}
Let
\[
  V=\bigoplus_{k\in\Z+\frac{1}{2}}\C\,\underline{k}.
\]
The half-infinite wedge space \(\Lambda^{\frac{\infty}{2}}V\) has a basis indexed by subsets
\[
  S=\{s_1>s_2>\cdots\}\subset \Z+\frac{1}{2}
\]
which differ from the vacuum set \(\Z_{\le0}-\frac{1}{2}\) in finitely many elements.  The basis vector is
\[
  v_S=\underline{s_1}\wedge \underline{s_2}\wedge\cdots.
\]
The vacuum vector is
\[
 v_\emptyset= |0\rangle
  =\underline{-1/2}\wedge\underline{-3/2}\wedge\underline{-5/2}\wedge\cdots.
\]
We use the inner product for which the wedge basis is orthonormal, and write
\(\langle0|\) for the covector dual to \(|0\rangle\).  For any operator \(\mathcal O\) on the infinite wedge space, we set
\[
  \langle 0|\mathcal O
  :=\bigl(\mathcal O^*|0\rangle\bigr)^\vee
\]
where \(\mathcal O^*\) is the adjoint of \(\mathcal O\) and \((\cdot)^\vee\) denotes the covector induced by the inner product.
The vacuum expectation of \(\mathcal O\) is then
\[
  \langle\mathcal O\rangle
  :=\langle0|\mathcal O|0\rangle
  =(\mathcal Ov_{\emptyset},v_{\emptyset}).
\]
For \(k\in\Z+\frac{1}{2}\), let \(\Psi_k\) be exterior multiplication by \(\underline{k}\), and let \(\Psi_k^*\) be its adjoint.  Define normally ordered products by
\[
 :\Psi_i\Psi_j^*:
 =
 \begin{cases}
 \Psi_i\Psi_j^*,&j>0,\\
 -\Psi_j^*\Psi_i,&j<0.
 \end{cases}
\]
This gives the standard projective representation of \(\mathfrak{gl}(V)\) on \(\Lambda^{\frac{\infty}{2}}V\).  
For $i,j\in \mathbb{Z}+\frac{1}{2}$, let $E_{ij}$ denote the standard basis element in $\mathfrak{gl}(V)$. We then define a projective representation by the following assignment:
\begin{equation*}
    E_{ij} \mapsto \ :\Psi_i \Psi_j^*:
\end{equation*}

Let
\[
  S_0=\left\{-\frac12,-\frac32,-\frac52,\ldots\right\}.
\]
The charge of \(v_S\) is
\[
  c(S)=|S\setminus S_0|-|S_0\setminus S|.
\]
The charge-zero subspace has a basis \(\{v_\lambda\}\) indexed by partitions:
\[
  v_\lambda
  =
  \underline{\lambda_1-\frac12}
  \wedge\underline{\lambda_2-\frac32}
  \wedge\underline{\lambda_3-\frac52}\wedge\cdots.
\]
On this subspace, the energy operator
\[
  H=\sum_{k\in\Z+\frac12}kE_{kk}
\]
satisfies
\[
  Hv_\lambda=|\lambda|v_\lambda.
\]

\subsection{The \texorpdfstring{\(\mathcal E\)}{E}-operators}

For \(j\in\Z\), define
\begin{equation*}
  \Ecal_j(z)=
  \sum_{k\in\Z+1/2}e^{z(k-j/2)}E_{k-j,k}
  +\frac{\delta_{j,0}}{e^{z/2}-e^{-z/2}}.
\end{equation*}
Let
\[
  \vars(z)=e^{z/2}-e^{-z/2},
  \qquad
  \Scal(z)=\frac{\vars(z)}{z}.
\]
The basic commutator is
\begin{equation}\label{eq:E-commutator}
  [\Ecal_j(z),\Ecal_k(w)]
  =\vars(jw-kz)\Ecal_{j+k}(z+w).
\end{equation}
The energy commutator is
\[
  [H,\Ecal_j(z)]=-j\Ecal_j(z).
\]
The adjoint satisfies
\[ 
  \Ecal_j(z)^*=\Ecal_{-j}(z).
\]
For \(j\ne0\), put
\[
  \alpha_j=\Ecal_j(0).
\]
Then
\[
  [\alpha_i,\alpha_j]=i\delta_{i+j,0}.
\]

\subsection{Orbifold \texorpdfstring{\(A\)}{A}-operators}

For \(0\le a\le r-1\), the Johnson orbifold operators are defined as follows.  If \(a=0\),
\begin{equation*}
  \boldsymbol A_{0/r}(z,u,t)
  =\frac1u\Scal(ruz)^{tz/r}
  \sum_{m\in\Z}
  \frac{(tz\Scal(ruz))^m}{(1+tz/r)_m}
  \Ecal_{rm}(uz).
\end{equation*}
If \(1\le a\le r-1\),
\begin{equation*}
  \boldsymbol A_{a/r}(z,u,t)
  =\frac{t^{a/r}}u\frac{z}{tz+a}
  \Scal(ruz)^{(tz+a)/r}
  \sum_{m\in\Z}
  \frac{(tz\Scal(ruz))^m}{(1+(tz+a)/r)_m}
  \Ecal_{rm+a}(uz).
\end{equation*}
The adjoint operator \(\boldsymbol A_{a/r}^*\) is obtained by taking the adjoint and replacing \(t\) by \(-t\).  We write \(\boldsymbol A_{a/r}[k]\) for the coefficient of \(z^{k+1}\).

\section{The orbifold cap operator formula}\label{sec:cap-formula}
We now derive the non-equivariant operator formula for the orbifold cap.

\subsection{Equivariant orbifold cap formula}
We first recall the orbifold formula needed in the cap
calculation.  Let \(\PP^1_{r,s}\) be the root stack of orders \(r\) and
\(s\) at \(0\) and \(\infty\). Let \(T=\C^*\) acts on the coarse
\(\PP^1\) with tangent weights \(1\) at \(0\) and \(-1\) at \(\infty\),
and lift this action to \(\PP^1_{r,s}\).  Write
\(\mathbf 0_{i/r}\) and \(\boldsymbol\infty_{j/s}\) for the corresponding equivariant Chen-Ruan classes
with fixed age. Following \cite{J}, the full $\C^*$-equivariant orbifold Gromov--Witten potential of \(\PP^1_{r,s}\) is denoted as
\begin{align*}
 \tau_{r,s}(\mathbf x,\mathbf x^*)
 \coloneqq&\sum_{g\in\Z}\sum_{d\geq0}u^{2g-2}q^d
 \left\langle\exp\left(
 \sum_{\substack{k\geq0\\0\leq i<r}}
 x_k(i)\tau_k(\mathbf 0_{i/r})
 +\sum_{\substack{k\geq0\\0\leq j<s}}
 x_k^*(j)\tau_k(\boldsymbol\infty_{j/s})
 \right)\right\rangle^{\bullet,T}_{g,d}
\end{align*}
where descendants are formed with the coarse
cotangent classes \(\bar\psi\) and  \(2g-2\) denotes minus the Euler characteristic of a possibly
disconnected domain, Johnson's operator formula \cite[Theorem B]{J} reads
\begin{equation}\label{eq:Johnson-formula}
\tau_{r,s}(\mathbf x,\mathbf x^*)=
\left\langle
 e^{\sum_{k,i}x_k(i)\boldsymbol A_{i/r}[k]}
 e^{t\alpha_r/(ur)}
 \left(\frac{q}{t^{1/r}(-t)^{1/s}}\right)^H
 e^{-t\alpha_{-s}/(us)}
 e^{\sum_{k,j}x_k^*(j)\boldsymbol A_{j/s}^*[k]}
\right\rangle .
\end{equation}
Here the
last \(A\)-operators are defined by replacing \(r\) with \(s\) and setting
\(\boldsymbol A_{j/s}^*(z,u,t)=\boldsymbol A_{j/s}(z,u,-t)^*\).

For $s=1$, we recover the target \(\PP^1_{r,1}\) for the orbifold cap. We further set
\[[0]=[0]_{0/r},\qquad [\infty]=[\infty]_{0/1}\]
to be the equivariant point classes associated to $[0]$ and $[\infty]$ respectively.

Fix twisted sectors
\(\mathbf a=(a_1,\ldots,a_p)\), \(1\leq a_j\leq r-1\), and a relative
condition \(\nu=(\nu_1,\ldots,\nu_{\ell(\nu)})\) at \(\infty\). Let
\(\mathbf y=(y_1,\ldots,y_p)\),
\(\mathbf z=(z_1,\ldots,z_n)\), and
\(\mathbf w=(w_1,\ldots,w_m)\).  We define 
\begin{align}
G^{T,\mathrm{st}}_{\mathbf a,\nu}(\mathbf y,\mathbf z,\mathbf w)
\coloneqq &
\sum_{\substack{k_1,\ldots,k_p\geq0\\
                l_1,\ldots,l_n\geq0\\
                q_1,\ldots,q_m\geq0}}
\left\langle
 \prod_{j=1}^{p}\tau_{k_j}([0]_{a_j/r})
 \prod_{i=1}^{n}\tau_{l_i}([0])
 \prod_{h=1}^{m}\tau_{q_h}([\infty])
 \ \middle|\ \nu
\right\rangle^{\bullet,T,\OrbCap_r}
\nonumber\\[-2mm]
&\hspace{14mm}\cdot
 \prod_{j=1}^{p}y_j^{k_j+1}
 \prod_{i=1}^{n}z_i^{l_i+1}
 \prod_{h=1}^{m}w_h^{q_h+1}
\label{eq:eq-cap-potential}
\end{align}
to be the generating series of the $\C^*$-equivariant Gromov-Witten of $\OrbCap_r$ with fixed twisted sectors $\mathbf a$ and relative condition $\nu$. Its degree is
\(|\nu|=\sum_{i=1}^{\ell(\nu)}\nu_i\). The genus
variable is suppressed because it can be recovered from the homogeneous grading.
The superscript ``\({\rm st}\)'' emphasizes that  unstable terms are not contained in $G^{T,\mathrm{st}}_{\mathbf a,\nu}$.

To state the operator formula for the orbifold cap, we need the following regularized operators. We set
\[
  D=t^{-H/r}.
\]
Since \([H,\Ecal_{rm+a}]=-(rm+a)\Ecal_{rm+a}\),
\[
  D^{-1}\Ecal_{rm+a}(z)D=t^{-m-a/r}\Ecal_{rm+a}(z).
\]
Define
\[
  \widehat A_a(z,t)=D^{-1}\boldsymbol A_{a/r}(z,1,t)D.
\]
Then
\begin{align}
\widehat A_0(z,t)
&=\Scal(rz)^{tz/r}
\sum_{m\in\Z}\frac{(z\Scal(rz))^m}{(1+tz/r)_m}\Ecal_{rm}(z),
\label{eq:A0hat}\\
\widehat A_a(z,t)
&=\frac{z}{tz+a}\Scal(rz)^{(tz+a)/r}
\sum_{m\in\Z}\frac{(z\Scal(rz))^m}{(1+(tz+a)/r)_m}\Ecal_{rm+a}(z)
\label{eq:Aahat}
\end{align}
for \(1\le a\le r-1\).  Also
\begin{equation}\label{eq:propagation-factor}
  D^{-1}e^{t\alpha_r/r}D=e^{\alpha_r/r}.
\end{equation}
Thus it becomes independent of \(t\) after regularization.

\smallskip
\noindent\emph{Johnson coefficient commutator.}
Put
\[
  \widehat A_{a,p}(t)=[z^{p+1}]\widehat A_a(z,t).
\]
By \cite[Lemma 8]{J}, we have
\begin{equation}\label{eq:johnson-commutator-support}
[\widehat A_{a,p}(t),\widehat A_{b,q}(t)]
=
(-1)^p\delta_{p+q,-1}
\begin{cases}
  t\,\mathrm{Id}, & a=b=0,\\[2mm]
  r^{-1}\mathrm{Id},
    & 1\leq a,b\leq r-1 \text{ and } a+b=r,\\[2mm]
  0, & \text{otherwise}.
\end{cases}
\end{equation}

For \(\mathbf w=(w_1,\ldots,w_m)\), let
\[
  E(\mathbf w,t)
  =
  \sum_{\pi\in\Part([m])}t^{m-\ell(\pi)}
  \prod_{B\in\pi}T_B\,\Ecal_0(W_B),
\]
where \(W_B=\sum_{i\in B}w_i\) and
\begin{equation*}
T_B=\left(\prod_{i\in B}w_i\right)W_B^{|B|-2}    
\end{equation*}
Here \(\Part([m])\) is the set of partitions of
\([m]=\{1,\ldots,m\}\). We set
\(E(\varnothing,t)=1\). Since the \(\Ecal_0\)-factors commute, the block
product is unordered. 

The normalized relative vector is
\begin{equation}\label{eq:relative-ket}
 |\nu\rangle
 =\frac{1}{\mathfrak z(\nu)}
  \prod_{j=1}^{\ell(\nu)}\alpha_{-\nu_j}|0\rangle,
 \qquad
 \mathfrak z(\nu)=|\Aut(\nu)|\prod_{j=1}^{\ell(\nu)}\nu_j
\end{equation}
If relative markings are ordered, the factor
\(|\Aut(\nu)|\) is omitted. We abbreviate
\[
  \left\langle\mathcal O\middle|\nu\right\rangle
  :=
  \langle0|\mathcal O|\nu\rangle.
\]

\begin{thm}[Stable equivariant orbifold cap formula]\label{thm:eq-cap}
\begin{equation}\label{eq:eq-cap-formula}
  G^{T,\mathrm{st}}_{\mathbf a,\nu}(\mathbf y,\mathbf z,\mathbf w)
  =
  \left[
    \left\langle
      \overrightarrow{\prod_j}\widehat A_{a_j}(y_j,t)
      \overrightarrow{\prod_i}\widehat A_0(z_i,t)
      e^{\alpha_r/r}
      E(\mathbf w,-t)
    \middle|\nu\right\rangle
  \right]_{+}.
\end{equation}
The arrows indicate increasing-label order. 
\end{thm}

\medskip
\noindent\textbf{Positive-part convention.}
Let \(F(\mathbf y,\mathbf z,\mathbf w;t)\) be a Laurent series in
\(\mathbf y,\mathbf z,\mathbf w\), whose coefficients may depend on the
equivariant parameter \(t\).  We define
\begin{equation*}
[F]_+
:=
\sum_{\substack{
\boldsymbol\alpha\in\mathbb Z_{>0}^{p},\;
\boldsymbol\beta\in\mathbb Z_{>0}^{n},\;
\boldsymbol\gamma\in\mathbb Z_{>0}^{m}}}
([\mathbf y^{\boldsymbol\alpha}
 \mathbf z^{\boldsymbol\beta}
 \mathbf w^{\boldsymbol\gamma}]F)\,
\mathbf y^{\boldsymbol\alpha}
\mathbf z^{\boldsymbol\beta}
\mathbf w^{\boldsymbol\gamma},
\end{equation*}
where
\(
[\mathbf y^{\boldsymbol\alpha}
 \mathbf z^{\boldsymbol\beta}
 \mathbf w^{\boldsymbol\gamma}]F
\)
denotes the coefficient of the indicated monomial in \(F\). If one of the lists of variables is empty, the corresponding positivity
condition is omitted. Thus \([F]_+\) is the projection onto monomials whose exponent is strictly positive in every variable, while no positivity condition is imposed on the equivariant parameter \(t\).

\begin{proof}[Proof of Theorem \ref{thm:eq-cap}]
The proof follows \cite[Section 4.2.1]{KW23}, especially Theorem~4.5 and its
$\infty$-relative version in Remark~4.6. The difference is that the root
order at $0$ is generalized from $r=1$ to an arbitrary fixed positive
integer $r$, while only the root order $s$ at $\infty$ tends to infinity.
We apply \eqref{eq:Johnson-formula}, rather than a separate localization
calculation.  Add an \(s\)-th root at \(\infty\), and represent a relative
part \(\nu_i\) by a marking in the sector of age \(\nu_i/s\) on
\(\PP^1_{r,s}\).  For fixed \(r\) and sufficiently large \(s\), the
resulting equivariant orbifold invariant is polynomial in \(s\), and its
constant term is the corresponding invariant of
\((\PP^1_{r,1},\infty)=\OrbCap_r\); this is the higher-genus
orbifold--relative correspondence at \(\infty\)
\cites{TY,FWY,FWY2}.  Consequently, a Laurent series whose positive
part is $G^{T,\mathrm{st}}_{\mathbf a,\nu}$ is obtained by selecting from Johnson's
formula \eqref{eq:Johnson-formula} the insertions
\[
 \boldsymbol A_{a_j/r}(y_j,1,t),\qquad
 \boldsymbol A_{0/r}(z_i,1,t),\qquad
 \boldsymbol A^*_{0/s}(w_h,1,t),\qquad
 \boldsymbol A^*_{\nu_i/s}[0],
\]
taking the coefficient of \(q^{|\nu|}\), and then extracting \([s^0]\).

The degree factor in
\eqref{eq:Johnson-formula} contains \(D=t^{-H/r}\). We set
\[
  L
  =
  \overrightarrow{\prod_j}\boldsymbol A_{a_j/r}(y_j,1,t)
  \overrightarrow{\prod_i}\boldsymbol A_{0/r}(z_i,1,t).
\]
Then
\[
 L e^{t\alpha_r/r}D
 =D\,(D^{-1}LD)\,(D^{-1}e^{t\alpha_r/r}D).
\]
Since $\langle0|D=\langle 0|$, conjugation gives
\[
  D^{-1}LD
  =
  \overrightarrow{\prod_j}\widehat A_{a_j}(y_j,t)
  \overrightarrow{\prod_i}\widehat A_0(z_i,t),
\]
and
\eqref{eq:propagation-factor} gives \(e^{\alpha_r/r}\). On the
coefficient of \(q^{|\nu|}\), the remaining degree factor is
\((-t)^{-|\nu|/s}\).  The selected relative-sector operators $\boldsymbol A^*_{\nu_i/s}[0]$ contribute
\[\prod_i(-t)^{\nu_i/s}=(-t)^{|\nu|/s}.\] 
So these factors cancel exactly
before the large-\(s\) coefficient is taken.

We next use the large-root analysis of \cite{KW23}.  Their structural
decomposition partitions the untwisted \(\infty\)-markings into connected
blocks, and their Hodge-integral formula identifies the scalar attached to
each block \cite[Lemmas 4.2 and 4.4]{KW23}. Hence the constant term of a block \(B\) is
\begin{equation}\label{eq:s-block-constant}
 (-t)^{|B|-1}
 \left(\prod_{i\in B}w_i\right)W_B^{|B|-2}\Ecal_0(W_B)
 =(-t)^{|B|-1}T_B\Ecal_0(W_B).
\end{equation}
Multiplying
\eqref{eq:s-block-constant} over the blocks and summing over set partitions
produces \(E(\mathbf w,-t)\).

Finally, after the preceding fractional-power cancellation and the large-$s$ analysis, the product \(\overrightarrow\prod_{i}\boldsymbol A^*_{\nu_i/s}[0]\) gives \(\prod_i\alpha_{-\nu_i}/\nu_i\). There is an additional factor \(|\Aut(\nu)|^{-1}\) due to the fact that our relative markings are unordered. So their product is
\[
 \frac{1}{\mathfrak z(\nu)}
 \prod_i\alpha_{-\nu_i},
\]
namely the relative vector \eqref{eq:relative-ket}. Substitution into
\eqref{eq:Johnson-formula} gives
\[
  \left\langle
  \overrightarrow{\prod_j}\widehat A_{a_j}(y_j,t)
  \overrightarrow{\prod_i}\widehat A_0(z_i,t)
  e^{\alpha_r/r}
  E(\mathbf w,-t)
  \middle|\nu\right\rangle.
\]
By the unstable convention of
\cite{J}, applying
\([\cdot]_{+}\) extracts precisely
the stable invariants in \eqref{eq:eq-cap-potential}, proving
\eqref{eq:eq-cap-formula}.

\end{proof}

\subsection{The non-equivariant limit}

In equivariant cohomology,
\[
 \id=\frac{[0]-[\infty]}{t}.
\]
Thus \(m\) insertions of \(\id\) are obtained by extracting the coefficient
of \(t^m\) from an inclusion--exclusion sum.  For \(S\subseteq[m]\), let
\(\mathbf w_S=(w_i)_{i\in S}\) in increasing order.  Define
\[
  M_r(\mathbf w;t)
  =
  \sum_{S\subseteq[m]}(-1)^{|S|}
  \left(\overrightarrow{\prod_{i\notin S}}\widehat A_0(w_i,t)\right)
  e^{\alpha_r/r}E(\mathbf w_S,-t).
\]
The arrow denotes increasing-label order.

For a nonempty block \(B\subseteq[m]\), put
\[
  b(B)=\max B,
  \qquad
  W_B=\sum_{i\in B}w_i,
\]
and define
\[
  A_B^{\vee,r}
  \coloneqq
  T_B\left[
    \frac{w_{b(B)}}{W_B}U(W_B)
    +
    \frac{w_{b(B)}}{r}
    \log\frac{W_B}{w_{b(B)}}P(W_B)
  \right],
\]
where
\[
  P(z)=[t^0]\widehat A_0(z,t),
  \qquad
  U(z)=[t^1]\widehat A_0(z,t).
\]

\begin{prop}\label{prop:M-redu}
The operator \(M_r\) satisfies
\begin{enumerate}[(i)]
\item \([t^q]M_r(\mathbf w;t)=0\) for \(q<m\);
\item
\[
  [t^m]M_r(\mathbf w;t)
  =
  \sum_{\pi\in\Part([m])}
  \overrightarrow{\prod_{B\in\pi}}A_B^{\vee,r}\,e^{\alpha_r/r}.
\]
\end{enumerate}
In the ordered product, the blocks are arranged by increasing maximal label.
\end{prop}

\begin{proof}
The proof follows the interaction-diagram formalism of \cite{OP3}, with two modifications specific to the present setting.

First, for each operator $\widehat A_0(w_i,t)$, we commute its  $t^0$-term $P(w_i)$ to the right. Its interaction with $\widehat A_0(w_j,t)$ follows from the two-parameter commutation relation \eqref{eq:P-two-param-comm} by setting $b=tw_j/r$ and using \( 
\widehat A_0(w_j,t) = \widehat{\mathcal A}(tw_j/r,w_j) 
\)
:
\begin{equation*}
[P(w_i),\widehat A_0(w_j,t)]=\vars(tw_iw_j) \left(1+\frac{w_i}{w_j}\right)^{tw_j/r} \widehat{\mathcal{A}}(tw_j/r,w_i+w_j).
\end{equation*}

Second, an interaction with the vertex $\infty$ is governed by the conjugation identity \begin{equation*} 
e^{-\alpha_r/r}P(w_i)e^{\alpha_r/r} = \Ecal_0(w_i)
\end{equation*}
which can be deduced from the basic commutator \eqref{eq:E-commutator}.

By the same inclusion-exclusion cancellation as in the proof of \cite[Proposition 4.1]{OP3}, only the principal terms, where $S = \emptyset$ and no interaction with $\infty$ occurs, survive. For these remaining principal terms, each interaction edge and terminal vertex contributes at least $t^1$, implying each connected block $B$ contributes at least $t^{|B|}$. Summing over all blocks in the partition of $[m]$ yields the minimal degree $t^m$, proving (i). For (ii), summing the lowest-order trees over $B$ yields the weighted Cayley tree polynomial $T_B$. Combining this with the $t^1$-coefficient of the terminal operator
\[
  \left(\frac{W_B}{w_{b(B)}}\right)^{tw_{b(B)}/r} \widehat{\mathcal{A}} \left(tw_{b(B)}/r,W_B\right),
\]
as given by Corollary \ref{cor:multi-P-two-param-comm}, precisely produces the factor $A^{\vee,r}_B$. 
\end{proof}

Write \(\phi_a\) for the twisted-sector class of age \(a/r\) at the
orbifold point \(0\).  Define the stable non-equivariant orbifold cap series by
\begin{align*}
  G^{\mathrm{st}}_{\mathbf a,\nu}(\mathbf y,\mathbf z,\mathbf w)
  \coloneqq{}&
  \sum_{\mathbf k,\mathbf l,\mathbf q\geq0}
  \left\langle
    \prod_j\tau_{k_j}(\phi_{a_j})
    \prod_i\tau_{l_i}(\omega)
    \prod_h\tau_{q_h}(\id)
  \ \middle|\ \nu
  \right\rangle^{\bullet,\OrbCap_r}\\
 &\hspace{12mm}\cdot
  \prod_jy_j^{k_j+1}
  \prod_i z_i^{l_i+1}
  \prod_h w_h^{q_h+1}.
\end{align*}
\begin{thm}[Stable non-equivariant orbifold cap formula]\label{thm:noneq-cap}
\begin{equation}\label{eq:non-eq-cap}
  G^{\mathrm{st}}_{\mathbf a,\nu}(\mathbf y,\mathbf z,\mathbf w)
  =
  \left[
    \left\langle
      \overrightarrow{\prod_j} T_{a_j}(y_j)
      \overrightarrow{\prod_i} P(z_i)
      \sum_{\pi\in\Part([m])}
      \overrightarrow{\prod_{B\in\pi}}A_B^{\vee,r}
      e^{\alpha_r/r}
    \middle|\nu\right\rangle
  \right]_{+},
\end{equation}
where \(T_a(y)=[t^0]\widehat A_a(y,t)\).
\end{thm}

\begin{proof}
Replacing each of the identity insertions by
\(([0]-[\infty])/t\) and applying Theorem~\ref{thm:eq-cap} gives
\[
  G^{\mathrm{st}}_{\mathbf a,\nu}
  =
  \left[
    [t^m]
    \left\langle
      \overrightarrow{\prod_j}\widehat A_{a_j}(y_j,t)
      \overrightarrow{\prod_i}\widehat A_0(z_i,t)
      M_r(\mathbf w;t)
    \middle|\nu\right\rangle
  \right]_{+}.
\]
By Proposition~\ref{prop:M-redu} (i), \(M_r(\mathbf w;t)\) has no term of
degree below \(m\).  Since the regularized \(A\)-operators are power series
in \(t\), only their constant terms can contribute to \([t^m]\).  Hence
\[
  [t^m]\left(
    \overrightarrow{\prod_j}\widehat A_{a_j}(y_j,t)
    \overrightarrow{\prod_i}\widehat A_0(z_i,t)
    M_r(\mathbf w;t)
  \right)
  =
  \overrightarrow{\prod_j}T_{a_j}(y_j)
  \overrightarrow{\prod_i}P(z_i)\,[t^m]M_r(\mathbf w;t).
\]
Proposition~\ref{prop:M-redu} (ii) now gives
\eqref{eq:non-eq-cap}.
\end{proof}

\section{Commutators and left-vacuum evaluations}\label{sec:comm-vac}

This section proves the operator identities used in the orbifold cap Virasoro proof.

\subsection{Basic commutators}
In the following, a factor
of the form
\[
  \left(1+\frac{x}{y}\right)^\lambda
  \quad\text{or}\quad
  \log\left(1+\frac{x}{y}\right)
\]
is always expanded in nonnegative powers of \(x/y\).  The expression with \(x\)
and \(y\) interchanged is expanded in nonnegative powers of \(y/x\).

Define the two parameter untwisted operator
\[
  \widehat {\mathcal{A}}(b,z)
\coloneqq \Scal(rz)^b\sum_{\ell\in\Z}\frac{(z\Scal(rz))^\ell}{(1+b)_\ell}\Ecal_{r\ell}(z),
  \label{eq:two-param-a0}
\]
Then \(\widehat A_0(z,t)=\widehat {\mathcal{A}}(tz/r,z)\). For
\(k\in\mathbb Z_{\geq0}\), we regard
\[
  \binom{x}{k}
  :=
  \frac{x(x-1)\cdots(x-k+1)}{k!}
\]
as a polynomial in \(x\).

\begin{lem}\label{lem:P-comm-two-param}
\begin{equation}\label{eq:P-two-param-comm}
  [P(x),\widehat {\mathcal{A}}(b,y)]
  =\vars(rbx)\left(1+\frac{x}{y}\right)^b\widehat {\mathcal{A}}(b,x+y).
\end{equation}
\end{lem}
\begin{proof}
For a fixed \(m\in\Z\), collect the terms proportional to
\(\Ecal_{rm}(x+y)\). By \eqref{eq:E-commutator}, its coefficient on the left-hand side is
\begin{align*}
  &\Scal(ry)^b
  \sum_{k\ge0}
  \frac{(x\Scal(rx))^k(y\Scal(ry))^{m-k}}
  {k!(1+b)_{m-k}}
  \vars\bigl(r(ky-(m-k)x)\bigr).
\end{align*}
Using
\begin{equation}\label{eq:pochhammer-binomial}
  \frac{(1+b)_m}{(1+b)_{m-k}}=k!\binom{m+b}{k},  
\end{equation}
this coefficient becomes
\[
  \frac{\Scal(ry)^b(y\Scal(ry))^m}{(1+b)_m}
  \sum_{k\geq0}
  \binom{m+b}{k}
  \left(\frac{\vars(rx)}{\vars(ry)}\right)^k
  \vars\bigl(r(ky-(m-k)x)\bigr).
\]
The elementary binomial identity
\[
  \sum_{k\geq0}
  \binom{m+b}{k}
  \left(\frac{\vars(rx)}{\vars(ry)}\right)^k
  \vars\bigl(r(ky-(m-k)x)\bigr)
  =
  \vars(rbx)
  \left(
    \frac{\vars(r(x+y))}{\vars(ry)}
  \right)^{m+b}
\]
transforms this expression into
\[
  \vars(rbx)
  \left(1+\frac{x}{y}\right)^b
  \Scal(r(x+y))^b
  \frac{((x+y)\Scal(r(x+y)))^m}{(1+b)_m}.
\]
This is the coefficient of \(\Ecal_{rm}(x+y)\) on the right-hand side of
\eqref{eq:P-two-param-comm}.
\end{proof}

\begin{cor}\label{cor:multi-P-two-param-comm}
\[ 
  [P(x_1),[P(x_2),[\dots [P(x_n),\widehat {\mathcal{A}}(b,y)]\dots]]]
  =\prod\varsigma(rbx_i)
  \left(1+\frac{\sum x_i}{y}\right)^b \widehat{\mathcal{A}}(b,\sum x_i+y).
\]
\end{cor}

For later coefficient
extraction, write
\begin{equation*}
  P(z)=\sum_{k\in\Z}P_kz^{k+1},\qquad
  U(z)=\sum_{k\in\Z}U_kz^{k+1},\qquad
  T_a(z)=\sum_{k\in\Z}T_{k,a}z^{k+1}.
\end{equation*}

\begin{cor}\label{cor:P-U-T-U}
The following commutators hold:
\begin{align}
  [P(x),U(y)]&=xyP(x+y),\label{eq:P-U-comm}\\
  [T_a(x),U(y)]&=xy\left(1+\frac{y}{x}\right)^{a/r-1}T_a(x+y).\label{eq:T-U-comm}
\end{align}
In particular, for $m,k\geq 0$, we have
\begin{align}
  [P_m,U_k]&=\frac{(m+1)_k}{k!}P_{m+k-1},\label{eq:P-U-coeff}\\
  [T_{m,a},U_k]&=\frac{(m+a/r)_k}{k!}T_{m+k-1,a}.
\label{eq:T-U-modes}
\end{align}
\end{cor}
\begin{proof}
    $[P(x),U(y)]$ is exactly the $t^1$ term of $[P(x),\widehat A_0(y,t)]$. By Lemma \ref{lem:P-comm-two-param}, 
\[
  [P(x),\widehat A_0(y,t)]=[P(x),\widehat {\mathcal A}(ty/r,y)]
  =\vars(txy)\left(1+\frac xy\right)^{ty/r}\widehat {\mathcal A}(ty/r,x+y).
\]
The $t^1$ term on the right-hand side gives \(xyP(x+y)\), which proves \eqref{eq:P-U-comm}.

For \eqref{eq:T-U-comm}, we introduce the two-parameter sector-\(a\) operator
\[
  \widehat{\mathcal A}_a(b,z)
  :=
  \frac{z}{r(b+\theta_a)}\Scal(rz)^{b+\theta_a}
  \sum_{m\in\Z}
  \frac{(z\Scal(rz))^m}{(1+b+\theta_a)_m}\Ecal_{rm+a}(z),
\]
where \(\theta_a=a/r\).  Since
\(\widehat A_a(x,t)=\widehat{\mathcal A}_a(tx/r,x)\),
\begin{equation}
  T_a(x)
  =
  \widehat A_a(x,0)
  =
  \widehat{\mathcal A}_a(0,x),
  \qquad
  T_a^{(1)}(x)
  :=
  [t]\widehat A_a(x,t)
  =
  \frac{x}{r}
  \left.
    \frac{\partial}{\partial b}\widehat{\mathcal A}_a(b,x)
  \right|_{b=0}.
  \label{eq:Ta1-from-b}
\end{equation}

By
\eqref{eq:johnson-commutator-support},
\[
  [\widehat A_a(x,t),\widehat A_0(y,t)]=0
  \qquad(1\leq a\leq r-1).
\]
Taking the coefficient of \(t\) gives
\begin{equation}
  [T_a(x),U(y)]
  =
  [P(y),T_a^{(1)}(x)].
  \label{eq:first-order-sector-commutativity}
\end{equation}
It remains to compute the right-hand side.  We claim that
\begin{equation}
  [P(y),\widehat{\mathcal A}_a(b,x)]
  =
  \frac{x}{x+y}\vars(rby)
  \left(1+\frac yx\right)^{b+\theta_a}
  \widehat{\mathcal A}_a(b,x+y).
  \label{eq:one-param-mixed}
\end{equation}
For fixed \(m\in\Z\), compare the coefficients of
\(\Ecal_{rm+a}(x+y)\).  Applying \eqref{eq:pochhammer-binomial} with
\(b\) replaced by \(b+\theta_a\), the coefficient on the left is
\[
  \frac{x\Scal(rx)^{b+\theta_a}}
       {r(b+\theta_a)(1+b+\theta_a)_m}
  \sum_{k\geq0}
  \binom{m+b+\theta_a}{k}
  (y\Scal(ry))^k
  (x\Scal(rx))^{m-k}
  \vars\bigl(rkx-r(m-k+\theta_a)y\bigr).
\]
The binomial identity used in the proof of
Lemma~\ref{lem:P-comm-two-param}, with
\((m,x,y)\) replaced by \((m+\theta_a,y,x)\), gives
\[
  \sum_{k\geq0}
  \binom{m+b+\theta_a}{k}
  \left(\frac{\vars(ry)}{\vars(rx)}\right)^k
  \vars\bigl(rkx-r(m-k+\theta_a)y\bigr)
  =
  \vars(rby)
  \left(
    \frac{\vars(r(x+y))}{\vars(rx)}
  \right)^{m+b+\theta_a}.
\]
Substitution gives the coefficient on the right-hand side of
\eqref{eq:one-param-mixed}, which completes the proof of the claim.

Differentiate \eqref{eq:one-param-mixed} at \(b=0\) and multiply by \(x/r\),
as in \eqref{eq:Ta1-from-b}.  Since \(\vars(0)=0\), only the derivative of
\(\vars(rby)\) contributes.  Hence
\begin{align*}
  [P(y),T_a^{(1)}(x)]
  &=
  \frac{x}{r}\frac{x}{x+y}(ry)
  \left(1+\frac yx\right)^{\theta_a}T_a(x+y)\\
  &=
  xy\left(1+\frac yx\right)^{\theta_a-1}T_a(x+y).
\end{align*}
Together with \eqref{eq:first-order-sector-commutativity}, this proves
\eqref{eq:T-U-comm}.  Extracting the coefficient of
\(x^{m+1}y^{k+1}\) gives the two coefficient commutators.
\end{proof}

\subsection{Left-vacuum evaluation formulas}

Set
\begin{equation}\label{eq:harmonic}
  h_k=\sum_{j=1}^{k}\frac1j,
  \qquad h_0=0.
\end{equation}
For \(k\ge2\), define
\begin{equation}\label{eq:Rk-def}
  \begin{aligned}
    R_k={}&-\frac{h_k}{r}P_{k-1}
    +\frac1{2rk!}\sum_{p=0}^{k-3}(p+1)!(k-p-2)!P_pP_{k-p-3}
    \\
    &+\frac r{2k!}\sum_{a=1}^{r-1}
    \sum_{p=0}^{k-2}\left(\frac{a}{r}\right)_{p+1}
    \left(1-\frac{a}{r}\right)_{k-p-1}
    T_{p,a}T_{k-p-2,a^\vee}.
  \end{aligned}
\end{equation}
Here \(a^\vee=r-a\); a sum with upper limit smaller than its lower limit is
understood to be zero.

\begin{prop}\label{prop:vacuum-removal}
For \(k\ge2\), 
\[
  \langle0|U_k=\langle0|R_k.
\]
For \(k=1\),
\begin{equation}\label{eq:U1-vacuum}
    \langle0|U_1=-\frac1r\langle0|P_0+\frac{r^2-1}{24r}\langle0|.
\end{equation}
\end{prop}

The proof is given in Appendix~\ref{app:vacuum}.  For $k\leq 3$, the proof
uses the Boson--Fermion correspondence, while the general case
follows from two auxiliary commutator identities and induction.

\section{Reactions and marked identity recursion}\label{sec:reaction-marked}

\subsection{Reaction table} After its leading derivative marks an insertion \(\tau_k(\id)\), the
remaining terms of \(L_{k-1}\) replace that insertion according to the
following table.  It will be used for \(k\geq2\), \(\ell\geq0\), and
\(1\leq a\leq r-1\); empty sums are zero.
\small
\begin{longtable}{>{\raggedright\arraybackslash}p{0.28\textwidth}p{0.62\textwidth}}
\toprule
Reaction & Coefficient and output \\
\midrule
\(\id+\id\to\id\) &
\(\displaystyle \tau_k(\id)\tau_l(\id)\mapsto
\frac{(l)_k}{k!}\tau_{k+l-1}(\id)\) \\
\addlinespace
\(\id+\id\to\omega\) &
\(\displaystyle \tau_k(\id)\tau_l(\id)\mapsto
\frac{\mathsf (l)_{k}}{rk!}\left(\sum_{j=l}^{k+l-1}\frac{1}{j}\right)\tau_{k+l-2}(\omega)\) \\
\addlinespace
\(\id+\omega\to\omega\) &
\(\displaystyle \tau_k(\id)\tau_l(\omega)\mapsto
\frac{(l+1)_k}{k!}\tau_{k+l-1}(\omega)\) \\
\addlinespace
\(\id+\phi_a\to\phi_a\) &
\(\displaystyle \tau_k(\id)\tau_l(\phi_a)\mapsto
\frac{(l+a/r)_k}{k!}\tau_{k+l-1}(\phi_a)\) \\
\addlinespace
\(\id\to\omega\) &
\(\displaystyle \tau_k(\id)\mapsto
-\frac1r\left(\sum_{j=1}^k\frac1j\right)\tau_{k-1}(\omega)\) \\
\addlinespace
\(\id\to\omega+\omega\) &
\(\displaystyle \tau_k(\id)\mapsto
\frac1{2rk!}\sum_{p=0}^{k-3}(p+1)!(k-p-2)!
\tau_p(\omega)\tau_{k-p-3}(\omega)\) \\
\addlinespace
\(\id\to\phi_a+\phi_{a^\vee}\) &
\(\displaystyle \tau_k(\id)\mapsto
\frac{r}{2k!}\sum_{p=0}^{k-2}
\begin{aligned}[t]
& (a/r)_{p+1}(1-a/r)_{k-p-1}\\
&\times\tau_p(\phi_a)\tau_{k-p-2}(\phi_{a^\vee})
\end{aligned}\) \\
\bottomrule
\end{longtable}
\normalsize

The first two rows are called Class I reactions.  The rest are called Class II reactions.

\subsection{Class I series}
Let
\[
  \scrC_N(w_1,\ldots,w_N)
  =
  \sum_{\pi\in\Part([N])}
  \overrightarrow{\prod_{B\in\pi}} A_B^{\vee,r},
  \qquad
  \scrC_0=1,
\]
where the blocks of each partition are written in increasing order of their
maximal labels.  For a nonempty block \(B\), recall that
\[
  W_B=\sum_{i\in B}w_i,
  \qquad
  b(B)=\max B,
  \qquad
  T_B=\left(\prod_{i\in B}w_i\right)W_B^{|B|-2},
\]
and put
\[
  A_B^{\vee,r}=I_B+E_B,
  \qquad
  I_B=T_B\frac{w_{b(B)}}{W_B}U(W_B),
  \qquad
  E_B=\frac{T_Bw_{b(B)}}{r}
       \log\frac{W_B}{w_{b(B)}}P(W_B).
\]
We call $I_B$ the identity part and $E_B$ the point part.

\subsection{Marked identity recursion}
\begin{lem}\label{lem:point-absorption}
Let \(L\) be a finite totally ordered set of identity labels with variables
\(\{w_i\}_{i\in L}\), and let \(*\) be a distinguished point label with
variable \(x\).  For a set partition
\(\pi=(B_1,\ldots,B_{q(\pi)})\in \Part(L)\), write its blocks in increasing order
of their maximal labels and set
\[
  X_0=x,
  \qquad
  X_\alpha=x+\sum_{\beta=1}^{\alpha}W_{B_\beta}.
\]
Then
\begin{equation}
  \sum_{\pi\in\Part(L)}
  \prod_{\alpha=1}^{q(\pi)}
  X_{\alpha-1}\,w_{b(B_\alpha)}\,T_{B_\alpha}
  =
  \left(x\prod_{i\in L}w_i\right)
  \left(x+\sum_{i\in L}w_i\right)^{|L|-1}.
  \label{eq:point-absorption}
\end{equation}
For \(L=\varnothing\), both sides are understood to be \(1\).
\end{lem}

\begin{proof}
Let $F_L(x)$ denote the left-hand side of \eqref{eq:point-absorption}, with $F_\emptyset(x)=1$. For $L \neq \emptyset$, let $m = \max L$. In an ordered partition of
$L$, the terminal block $B$ must contain $m$, while the preceding blocks partition $L \setminus B$. Thus $F_L(x)$ gives the recursion
\begin{equation}
	F_L(x)=
	\sum_{\substack{B\subseteq L\\ m\in B}}
	F_{L\setminus B}(x)
	\bigl(x+W_{L\setminus B}\bigr)w_mT_B .
	\label{eq:FL-recursion}
\end{equation}

Now let
\[
  G_L(x)\coloneqq
  \sum_{\Gamma\in\Tree(\{*\}\cup L)}
  \prod_{\{u,v\}\in E(\Gamma)}x_ux_v,
  \qquad 
  x_*=x,
  \quad 
  x_i=w_i.
\]
Given such a tree, remove the first edge on the path from \(m\) to \(*\).
If \(B\) is the vertex set of the component containing \(m\), then
\(m\in B\subseteq L\).  Conversely, the tree is reconstructed uniquely by
choosing a tree on \(B\), a tree on
\(\{*\}\sqcup(L\setminus B)\), a vertex in the latter tree, and then joining
that vertex to \(m\).  Summing over the choice of the joining vertex gives
\((x+W_{L\setminus B})w_m\).
Since \(T_B\) is the weighted tree polynomial on \(B\), \(G_L(x)\) satisfies
the same recursion \eqref{eq:FL-recursion}.  The common initial value
at \(L=\varnothing\) therefore gives \(F_L(x)=G_L(x)\).
	
The explicit product form \eqref{eq:point-absorption} then follows immediately from the classical weighted Cayley tree identity:
\[
  G_L(x)
  =
  \left(x\prod_{i\in L}w_i\right)
  \left(x+\sum_{i\in L}w_i\right)^{|L|-1}.
\]
\end{proof}

\begin{thm}[Marked identity recursion]\label{thm:marked-recursion}
Let \(z=w_1\) be the marked identity variable and set
\[
  \widehat{\mathbf w}_j=(w_2,\ldots,\widehat{w_j},\ldots,w_N).
\]
Then
  \begin{equation}\label{eq:marked-recursion}
  \begin{aligned}
    \scrC_N(z,w_2,\ldots,w_N)
    ={}&U(z)\scrC_{N-1}(w_2,\ldots,w_N)
    \\
    &+\sum_{j=2}^{N}\frac{zw_j^2}{z+w_j}
    \scrC_{N-1}(w_2,\ldots,z+w_j,\ldots,w_N)
    \\
    &+\sum_{j=2}^{N}\frac{zw_j^2}{r}
    \log\left(1+\frac z{w_j}\right)
    P(z+w_j)\scrC_{N-2}(\widehat{\mathbf w}_j).
  \end{aligned}
  \end{equation}
\end{thm}

\begin{rmk}
All expressions involving composite variables are interpreted in the global
iterated Laurent expansion corresponding to
\[
  |z|\ll |w_2|\ll\cdots\ll |w_N|.
\]
In particular, when \(P(z+w_j)\) is moved to its canonical position, the
commutator terms are expanded in this order.
\end{rmk}

\begin{proof}[Proof of Theorem \ref{thm:marked-recursion}]
Let
\begin{equation}
	K_I(z,w)=\frac{zw^2}{z+w},\qquad K_E(z,w)=\frac{zw^2}{r}\log\left(1+\frac zw\right).
	\label{eq:KI-KE}
\end{equation}
We first describe the two types of summands on the RHS of \eqref{eq:marked-recursion}.  A term involving \(K_I(z,w_j)\) will be
called a \(K_I\)-contraction term: it contracts the marked label \(1\) with
the label \(j\) by replacing \(w_j\) with \(z+w_j\) in the block containing
\(j\).  A term involving \(K_E(z,w_j)\) will be called a
\(K_E\)-absorption term: it first creates the point operator
\(P(z+w_j)\), which may subsequently absorb other blocks through the
point--identity commutator \eqref{eq:P-U-comm}.

Expand the LHS of \eqref{eq:marked-recursion} by ordered set partitions, and let \(B_\bullet\)
denote the block containing the marked label \(1\).  If \(B_\bullet=\{1\}\), it is the first block with respect to the increasing-maximal order and
\[
  A_{\{1\}}^{\vee,r}=U(z).
\]
Summing over the remaining blocks gives
\[
  U(z)\scrC_{N-1}(w_2,\ldots,w_N),
\]
which is the first line on the RHS of \eqref{eq:marked-recursion}.

It remains to reconstruct a fixed nonsingleton marked block
\(B=B_\bullet\).  Set
\[
  b=b(B),\qquad W=W_B.
\]
Fix a partition \(\rho\) of the exterior labels \([N]\setminus B\), and
write \(\rho_{<b}\) and \(\rho_{>b}\) for the blocks of \(\rho\) whose maximal
labels are, respectively, smaller and larger than \(b\).  Put
\[
  A_{\rho,<b}
  =
  \overrightarrow{\prod_{C\in\rho_{<b}}}A_C^{\vee,r},
  \qquad
  A_{\rho,>b}
  =
  \overrightarrow{\prod_{C\in\rho_{>b}}}A_C^{\vee,r}.
\]
The target term associated with \((B,\rho)\) is
\begin{equation}\label{eq:target}
  A_{\rho,<b}\,A_B^{\vee,r}\,A_{\rho,>b}.
\end{equation}
We will reconstruct identity and point parts of \(A_B^{\vee,r}\) while keeping $A_{\rho,<b},\,A_{\rho,>b}$ unchanged.

For \(j\in B\setminus\{1\}\), let \(B^{(j)}\) be the block obtained by deleting
the label \(1\) and replacing the variable \(w_j\) by \(z+w_j\).  This
contraction preserves both the total block variable \(W\) and the 
position of the block.

\smallskip
\noindent\emph{Identity part.}
Direct substitution into the definition of \(I_B\) gives
\[
  K_I(z,w_j)I_{B^{(j)}}
  =
  \begin{cases}
    \dfrac{w_j}{W}I_B, & j\ne b,\\[6pt]
    \dfrac{z+w_b}{W}I_B, & j=b.
  \end{cases}
\]
Therefore
\begin{align}
  \sum_{j\in B\setminus\{1\}}
  K_I(z,w_j)I_{B^{(j)}}
  &=
  \left(
    \sum_{\substack{j\in B\setminus\{1,b\}}}\frac{w_j}{W}
    +\frac{z+w_b}{W}
  \right)I_B
  \notag\\
  &=I_B. 
  \label{eq:identity-output-splitting}
\end{align} 

After inserting \eqref{eq:identity-output-splitting} between
\(A_{\rho,<b}\) and \(A_{\rho,>b}\), these give the identity part of 
\eqref{eq:target}.

\smallskip
\noindent\emph{Point part.}
For the fixed block \(B\), put
\[
  x=z+w_b,
  \qquad
  L=B\setminus\{1,b\}.
\]
Then
\begin{equation}
  W=x+\sum_{i\in L}w_i,
  \qquad
  E_B
  =
  \frac{w_bT_B}{r}\log\frac{W}{w_b}\,P(W).
  \label{eq:fixed-point-block}
\end{equation}

We begin with the \(K_I\)-contraction terms.  If \(j\in L\), the terminal
label remains \(b\), and direct substitution gives
\begin{equation}
  K_I(z,w_j)E_{B^{(j)}}
  =
  \frac{w_j}{W}E_B.
  \label{eq:nonterminal-point-contraction}
\end{equation}
For \(j=b\), the terminal variable changes from \(w_b\) to \(x\), so
\begin{equation}
  K_I(z,w_b)E_{B^{(b)}}
  =
  \frac xW\,\frac{w_bT_B}r
  \log\frac{W}{x}\,P(W).
  \label{eq:terminal-KI}
\end{equation}
Thus the \(j=b\) contraction accounts for the first term in
\begin{equation}
  \log\frac{W}{w_b}
  =
  \log\frac{W}{x}
  +
  \log\frac{x}{w_b}.
  \label{eq:log-split}
\end{equation}
 We now show that the
missing second term is supplied by a \(K_E\)-absorption term.

A \(K_E\)-absorption term indexed by \(j\) starts with the marked point
operator \(P(z+w_j)\).  To place it in increasing-maximal order, it is moved
past precisely those blocks \(C\) with \(b(C)<j\).  Whenever the second term on the RHS
of 
\begin{equation}
  P(X)A_C^{\vee,r}
  =
  A_C^{\vee,r}P(X)
  {}+Xw_{b(C)}T_C\,P(X+W_C)
  \label{eq:commute-through-block}
\end{equation}
is chosen, \(C\) is absorbed and the point
variable changes from \(X\) to \(X+W_C\). Here, \([P(X),E_C]=0\), so the second term in
\eqref{eq:commute-through-block} comes from the commutator \eqref{eq:P-U-comm} with \(I_C\). All absorbed blocks have maximal
label smaller than \(j\); hence the resulting marked block still has maximal
label \(j\).  Since the fixed block \(B\) has maximal label \(b\), it can be
produced only by the summand indexed by \(j=b\):
\begin{equation}
  K_E(z,w_b)P(x)
  \scrC_{N-2}(\widehat{\boldsymbol w}_b).
  \label{eq:relevant-KE-term}
\end{equation}

We compute the part of \eqref{eq:relevant-KE-term} with the final marked block being precisely \(B\) and the exterior partition fixed to \(\rho\). Let
\[
  \pi=(C_1,\ldots,C_{q(\pi)})
\]
be the partition of \(L\), with its blocks in increasing order of their
maximal labels.  Exterior blocks whose maxima are smaller than \(b\) are
commuted through using the first term of \eqref{eq:commute-through-block}; exterior
blocks whose maxima are larger than \(b\) already lie after the marked block.
Consequently, all exterior blocks retain their relative order.

To obtain \(B\), every block \(C_\alpha\) of \(\pi\) must instead be absorbed
using the second term of \eqref{eq:commute-through-block}.  
The contribution associated with \(\pi\) is therefore
\begin{equation}
  \begin{gathered}
    A_{\rho,<b}\,\Xi_\pi\,A_{\rho,>b},\\
    \Xi_\pi
    :=
    \left(
      \prod_{\alpha=1}^{q(\pi)}
      X_{\alpha-1}w_{b(C_\alpha)}T_{C_\alpha}
    \right)P(W).
  \end{gathered}
  \label{eq:local-absorption-contribution}
\end{equation}
Summing
\eqref{eq:local-absorption-contribution} over all partitions of \(L\) and
applying Lemma~\ref{lem:point-absorption} gives
\begin{equation}
  A_{\rho,<b}
  \left(x\prod_{i\in L}w_i\right)W^{|L|-1}P(W)
  A_{\rho,>b}.
  \label{eq:absorbed-point-sum}
\end{equation}

Finally,
\[
  K_E(z,w_b)
  =
  \frac{zw_b^2}{r}\log\frac{x}{w_b},
  \qquad
  T_B
  =
  zw_b\left(\prod_{i\in L}w_i\right)W^{|L|}.
\]
Multiplying the middle factor in \eqref{eq:absorbed-point-sum} by
\(K_E(z,w_b)\) therefore yields
\begin{equation}
  \frac{x}{W}\,
  \frac{w_bT_B}{r}
  \log\frac{x}{w_b}\,P(W).
  \label{eq:terminal-KE}
\end{equation}
This is precisely the part missing from \eqref{eq:terminal-KI}.  By
\eqref{eq:log-split}, the \(K_I\)-contraction and \(K_E\)-absorption terms
indexed by \(b\) combine to give
\[
  \frac{x}{W}E_B.
\]
Adding the $K_I$-contractions \eqref{eq:nonterminal-point-contraction} indexed by
\(j\in L\), we recover
\begin{align*}
  \frac{x}{W}E_B
  +
  \sum_{j\in L}\frac{w_j}{W}E_B
  &=
  \frac{x+\sum_{j\in L}w_j}{W}E_B\\
  &=E_B.
\end{align*}
Thus the \(K_I\)-contraction and \(K_E\)-absorption terms reconstruct the point part of \eqref{eq:target}.

We have now reconstructed both
\(A_{\rho,<b}I_BA_{\rho,>b}\) and
\(A_{\rho,<b}E_BA_{\rho,>b}\) for every nonsingleton marked block \(B\) and
every exterior partition \(\rho\), while the singleton marked block gives the
first line of \eqref{eq:marked-recursion}. Hence every term on the LHS of \eqref{eq:marked-recursion} is obtained, and the marked identity recursion follows.
\end{proof}

\begin{cor}\label{cor:class-I-reactions}
For \(k\geq2\) and \(\ell\geq0\), the second and third lines of
\eqref{eq:marked-recursion} give
\[
  \tau_k(\id)\tau_\ell(\id)
  \longmapsto
  \frac{(\ell)_k}{k!}\tau_{k+\ell-1}(\id)
\]
and
\[
  \tau_k(\id)\tau_\ell(\id)
  \longmapsto
  \frac1{rk!}
  \left.\frac{d}{ds}(s)_k\right|_{s=\ell}
  \tau_{k+\ell-2}(\omega).
\]
\end{cor}
\begin{proof}
By Theorem~\ref{thm:noneq-cap} and~\ref{thm:marked-recursion}, the second
line merges the marked identity insertion with one other identity insertion
and again produces an identity insertion.  Its coefficient is
\[
  [z^{k+1}w^{\ell+1}]
  \frac{zw^2}{z+w}U(z+w)
  =
  \binom{k+\ell-1}{k}U_{k+\ell-1}
  =
  \frac{(\ell)_k}{k!}U_{k+\ell-1}.
\]

The third line produces a point insertion.  Setting \(x=z/w\) gives
\begin{align*}
  &[z^{k+1}w^{\ell+1}]
  \frac{zw^2}{r}
  \log\left(1+\frac zw\right)P(z+w)\\
  &\qquad=
  \frac1r[x^k](1+x)^{k+\ell-1}\log(1+x)\,
  P_{k+\ell-2}\\
  &\qquad=
  \frac1r
  \left.
    \frac{d}{ds}\binom{k+\ell-1+s}{k}
  \right|_{s=0}
  P_{k+\ell-2}\\
  &\qquad=
  \frac1{rk!}
  \left.\frac{d}{ds}(s)_k\right|_{s=\ell}
  P_{k+\ell-2}.
\end{align*}
For \(\ell\geq1\),
\[
  \left.\frac{d}{ds}(s)_k\right|_{s=\ell}
  =
  (\ell)_k\sum_{q=\ell}^{k+\ell-1}\frac1q,
\]
whereas at \(\ell=0\) it equals \((k-1)!\).  Thus the polynomial derivative
also covers the boundary case \(\ell=0\).
\end{proof}

\section{First orbifold cap proof}\label{sec:proof}

For a relative condition \(\nu\) at \(\infty\), we write
\[
  Z^{\OrbCap_r}[\nu]
\]
for the disconnected descendant potential of
\(\OrbCap_r=(\PP^1_{r,1},\infty)\).

\begin{thm}\label{thm:cap-virasoro}
For every \(r\ge1\) and every relative condition \(\nu\),
\[
  L_kZ^{\OrbCap_r}[\nu]=0,
  \qquad k\ge -1.
\]
\end{thm}
\begin{proof}
The logarithmic Euler characteristic 
\[
  \chi_{\log}(\OrbCap_r)
  =
  2-1-\left(1-\frac1r\right)
  =
  \frac1r.
\]
We prove the constraint coefficientwise in the descendant variables.  The
string equation is exactly \(L_{-1}Z^{\OrbCap_r}[\nu]=0\).

Fix \(k\ge2\).  The leading term of \(L_{k-1}\) is
\[
  -k!\frac{\partial}{\partial t_k^0}.
\]
Consequently, the coefficientwise equation
\(L_{k-1}Z^{\OrbCap_r}[\nu]=0\) is equivalent, after moving this leading
derivative to the other side and dividing by \(k!\), to the seven-reaction
identity in Section~\ref{sec:reaction-marked}.  We now prove that identity
from the orbifold cap operator formula.

Apply Theorem~\ref{thm:marked-recursion} to the identity factors in
\eqref{eq:non-eq-cap}, with the variable \(z\) assigned to the marked
insertion.  Corollary~\ref{cor:class-I-reactions} identifies the second and
third lines of the recursion with
\[
  \id+\id\longrightarrow\id,
  \qquad
  \id+\id\longrightarrow\omega.
\]
It remains to analyze the first line, in which the marked insertion survives
as \(U_k\).

Let \(S_1,\ldots,S_M\) be the coefficient operators lying to the left of
this \(U_k\).  Each \(S_j\) is either \(P_\ell\) or \(T_{\ell,a}\).  Moving
\(U_k\) to the left gives the identity
\begin{equation}\label{eq:stationary-telescope}
  S_1\cdots S_MU_k
  =
  U_kS_1\cdots S_M
  +
  \sum_{j=1}^{M}
  S_1\cdots S_{j-1}[S_j,U_k]S_{j+1}\cdots S_M.
\end{equation}
By Corollary~\ref{cor:P-U-T-U}, the commutators in
\eqref{eq:stationary-telescope} give, with the coefficients displayed in
the reaction table,
\[
  \id+\omega\longrightarrow\omega,
  \qquad
  \id+\phi_a\longrightarrow\phi_a.
\]
The first term on the right of \eqref{eq:stationary-telescope} places
\(U_k\) next to the left vacuum.  Proposition~\ref{prop:vacuum-removal}
then replaces it by \(R_k\).  Every coefficient index occurring in \(R_k\)
is nonnegative, so Lemma~\ref{lem:nonnegative-stationary-commute} allows
the new \(P\)- and \(T\)-factors to be restored to the order prescribed in
\eqref{eq:non-eq-cap}, without producing further commutators.  The three
summands of \(R_k\) therefore give
\[
  \id\longrightarrow\omega,
  \qquad
  \id\longrightarrow\omega+\omega,
  \qquad
  \id\longrightarrow\phi_a+\phi_{a^\vee}.
\]
The factor \(1/r=\chi_{\log}(\OrbCap_r)\) agrees with the corresponding
coefficients of \(L_{k-1}\).  Hence the operator formula reproduces every
term of \(L_{k-1}\), with no additional contribution, and therefore
\[
  L_{k-1}Z^{\OrbCap_r}[\nu]=0
  \qquad(k\ge2).
\]
In particular, the case \(k=2\) gives the \(L_1\)-constraint.  Since the
Virasoro operators satisfy
\[
  [L_1,L_{-1}]=2L_0,
\]
the \(L_{-1}\)- and \(L_1\)-constraints imply the \(L_0\)-constraint.
This proves the theorem.
\end{proof}


{\large \part{Second orbifold cap proof: degeneration and cap inversion}\label{part:cap-inversion-proof}}

\section{Second orbifold cap proof}
\label{sec:cap-inversion}

In this section, we provide another orbifold cap proof (Theorem \ref{thm:orbifold-cap2}). We approach this result by cap reversion (Theorem \ref{thm:cap-inversion}).

Let
\[
\mathrm{Cap}=(\PP^1,\{p\})
\]
be the ordinary relative cap. The cap reversion (Theorem \ref{thm:cap-inversion}) relies on a full-column-rank property of $\mathrm{Cap}$, which was pointed out and used in the proof of \cite[Proposition~1.3]{OP3}. We provide the proof here for convenience of readers.

\begin{lem}
\label{lem:stationary-cap-rank}
For \(d\ge0\), let $M_d$ be a matrix defined by
\[
M_d=
\left(
\left\langle
\mathbf L
\ \middle|\
\lambda
\right\rangle^{\bullet,\,\mathrm{Cap}}_d
\right)_{\mathbf L,\,\lambda\vdash d},
\]
where the rows are indexed by ordered absolute stationary insertion lists
\[
\mathbf L=\begin{cases}
\bigl(
\tau_{\ell_1}(\omega),\ldots,\tau_{\ell_s}(\omega)
\bigr),&s\geq0,\\
\varnothing,&s=0.
\end{cases}
\]
Then the matrix \(M_d\) has full column rank.
\end{lem}

\begin{proof}
For $d=0$, the matrix \(M_0\) has only one column, and it has one entry
\[
\left\langle
\tau_0(\omega)
\ \middle|\
\varnothing
\right\rangle^{\bullet,\,\mathrm{Cap}}_0
=
\left\langle
\tau_0(\omega)
\ \middle|\
\varnothing
\right\rangle^{\mathrm{Cap}}_{g=1,d=0}
=
-\frac{1}{24}\neq 0.
\]
This proves that \(M_0\) has full column rank.

Now suppose that $d>0$. Let \(\mathcal P(d)\) be the set of unordered partitions of \(d\), so that
\(|\mathcal P(d)|=p(d)\).  From \cite[Theorem~1]{OP1}, we have
\begin{align}\label{formula-ordinarycapinvariant}
\left\langle
\mathbf L
\ \middle|\
\lambda
\right\rangle^{\bullet,\,\mathrm{Cap}}_d
=
\sum_{\rho\vdash d}
\left(\frac{\dim\rho}{d!}\right)^2
\prod_{i=1}^{s}
\frac{\mathbf p_{\ell_i+1}(\rho)}{(\ell_i+1)!}\,
\mathbf f_\lambda(\rho).
\end{align}
Here for $\ell>0$, the number $\mathbf p_\ell$ is defined by \[\mathbf p_{\ell}(\rho)=\sum_{i\ge1}[(\rho_i-i+\frac12)^\ell-(-i+\frac12)^\ell]+\ell!c_{\ell+1},\] with $c_{\ell+1}$ defined by 
\(\sum_{j\ge0}c_jz^j=\frac{z/2}{\sinh(z/2)}\). The number $\mathbf f_\lambda(\rho)$ is defined by
\(\mathbf f_\lambda(\rho)
=
|C_\lambda|
\frac{\chi^\rho_\lambda}{\dim\rho},
\)
where \(C_\lambda\subset S(d)\) is the conjugacy class of cycle type
\(\lambda\), and \(\chi^\rho_\lambda\) is the character of any element of $C_\lambda$ in the representation of $S(d)$ corresponding to \(\rho\).

Define matrices $A_d, D_d, H_d$ by
\[
(A_d)_{\mathbf L,\rho}
=
\prod_{i=1}^{s}
\frac{\mathbf p_{\ell_i+1}(\rho)}{(\ell_i+1)!},
\qquad
(D_d)_{\rho,\rho}
=
\left(\frac{\dim\rho}{d!}\right)^2,
\qquad
(H_d)_{\rho,\lambda}
=
\mathbf f_\lambda(\rho),
\]
where $D_d$ is diagonal. Note that \eqref{formula-ordinarycapinvariant} gives the matrix factorization
\[
M_d=A_dD_dH_d.
\]
The diagonal matrix \(D_d\) is invertible.  Note that
\[
H_d
=
\operatorname{diag}_{\rho\vdash d}
\left(\frac{1}{\dim \rho}\right)
\cdot
\left(\chi_{\lambda}^{\rho}\right)_{\rho\vdash d,\lambda\vdash d}
\cdot
\operatorname{diag}_{\lambda\vdash d}
\left(\lvert C_{\lambda}\rvert\right).
\]
Since the character table of any finite group is invertible, it follows that \(H_d\) is invertible. Consequently,
\[
\rank M_d=\rank A_d.
\]

Note that each column $\bigl(\mathbf f_\lambda(\rho)\bigr)_{\rho\vdash d}$ of $H_d$ can be viewed as in $\mathbb Q^{\mathcal P(d)}$, i.e. the space of functions of $\mathcal P(d)$. Since $H_d$ is invertible and the number of its columns is exactly $p(d)=\dim_\mathbb Q\mathbb Q^{\mathcal P(d)}$, it follows that the column vectors
\[
\bigl(\mathbf f_\lambda(\rho)\bigr)_{\rho\vdash d},
\qquad
\lambda\vdash d,
\]
form a basis of \(\mathbb Q^{\mathcal P(d)}\). Now, from \cite[\S0.4]{OP1}, \(\mathbf f_\lambda\in
\Lambda^*
=
\mathbb Q[\mathbf p_1,\mathbf p_2,\ldots]
\), and therefore, each column vector
\(
\bigl(\mathbf f_\lambda(\rho)\bigr)_{\rho\vdash d}
\)
is a $\Q$-linear combination of rows of \(A_d\). Therefore the rows of \(A_d\) span
\(\mathbb Q^{\mathcal P(d)}\), implying $\rank A_d\ge\dim_\mathbb Q\mathbb Q^{\mathcal P(d)}=p(d)$. So we get
\[
\rank M_d=\rank A_d\ge p(d).
\]
We finish the proof by noting that the number of columns of $M_d$ is $p(d)$. \end{proof}

\begin{thm}
\label{thm:cap-inversion}
 As in Section \ref{sec:degeneration}, consider the degeneration
\[
(\mathcal C,Q)
\rightsquigarrow
(\mathcal C',Q'\cup\{p\})
\cup_p
(\mathcal C_0,Q_0\cup\{p\}),
\]
in the special case $\mathcal C'=\mathcal C,Q'=Q,\mathcal C_0=\mathbb P^1,Q_0=\varnothing$. 
Fix $k\ge-1$, \(d\geq0\) and relative profiles
\(\boldsymbol{\eta}=(\eta_1,\dots,\eta_m)\) for $(\mathcal C,Q)$. 
Assume that
\[
L_{k}Z^{\cC,Q}_d[\boldsymbol\eta]=0.
\]
Then
\[
L_{k}Z^{\cC,Q\cup\{p\}}_d[\boldsymbol\eta,\lambda]=0,\quad \lambda\vdash d.
\]
\end{thm}

\begin{proof}
It suffices to prove that, for all ordered absolute insertion list $\mathbf B$, we have
\(\Delta_{k,d}^{C,Q\cup\{p\}}
\left(
\mathbf B
\ \middle|\
\boldsymbol\eta,\lambda
\right)
=0.\) Write
\[
\mathbf B=
\prod_{i=1}^{u}\tau_{a_i}(\id)\,\mathbf B^\circ,
\]
where \(\mathbf B^\circ\) contains no unit insertion. We argue by induction on the number \(u\).

Let \(\mathbf L\) be an arbitrary ordered absolute stationary insertion list, and consider the absolute insertion for $(\mathcal C,Q)$:
\[
\prod_{i=1}^{u}\tau_{a_i}(\id)\,\mathbf L\,\mathbf B^\circ.
\]
We choose specializations of insertions to $H^*_{\CR}(\mathcal C')\times H^*_{\CR}(\cC_0)=H^*_{\CR}(\cC)\times H^*(\mathbb P^1)$ as follows.
We specialize point classes in $\mathbf L$ to $\mathcal C_0$, and all insertions in
\(\mathbf B^\circ\) to $\cC'$.

We apply Proposition~\ref{prop:virasoro-defect-degeneration}. Note that
the defect of \(\mathcal C/Q\) vanishes by hypothesis, and every defect of the
ordinary cap vanishes by \cite{OP3}. Therefore,
\begin{equation}
\label{eq:cap-inversion-triangular-system}
\begin{aligned}
0
={}&
\sum_{\lambda\vdash d}
\mathfrak z(\lambda)
\sum_{S\subset\{1,\ldots,u\}}
\Delta_{k,d}^{\mathcal C/(Q\cup\{p\})}
\left(
\prod_{i\notin S}\tau_{a_i}(\id)\,
\mathbf B^\circ
\ \middle|\
\boldsymbol{\eta},\lambda
\right)\cdot
\left\langle
\prod_{i\in S}\tau_{a_i}(\id)\,
\mathbf L
\ \middle|\
\lambda
\right\rangle^{\bullet,\,\mathrm{Cap}}_d.
\end{aligned}
\end{equation}

When \(u=0\), we get
\begin{equation*}
0=
\sum_{\lambda\vdash d}
\mathfrak z(\lambda)
\Delta_{k,d}^{\cC/(Q\cup\{p\})}
\left(
\mathbf B
\ \middle|\boldsymbol{\eta},\lambda
\right)\cdot\left\langle
\mathbf L
\ \middle|\
\lambda
\right\rangle^{\bullet,\,\mathrm{Cap}}_d
\end{equation*}
for every stationary list \(\mathbf L\). Note that \(\mathfrak z(\lambda)\neq0\) for every \(\lambda\), and therefore
multiplying each column of \(M_d\) by the corresponding
\(\mathfrak z(\lambda)\) preserves the column rank of the matrix. So by Lemma \ref{lem:stationary-cap-rank}, we have
\[
\Delta_{k,d}^{\mathcal C/(Q\cup\{p\})}
\left(
\mathbf B
\ \middle|\
\boldsymbol{\eta},\lambda
\right)
=0.
\]

Assume that $u>0$. If \(S\neq\varnothing\), the defect in
\eqref{eq:cap-inversion-triangular-system} contains
\(u-|S|<u\) absolute unit insertions and therefore vanishes by the induction
hypothesis. If $S=\varnothing$, we get
\begin{equation*}
\begin{aligned}
0
={}&
\sum_{\lambda\vdash d}
\mathfrak z(\lambda)
\Delta_{k,d}^{\mathcal C/(Q\cup\{p\})}
\left(
\prod_{i=1}^u\tau_{a_i}(\id)\,
\mathbf B^\circ
\ \middle|\
\boldsymbol{\eta},\lambda
\right)\cdot
\left\langle
\mathbf L
\ \middle|\
\lambda
\right\rangle^{\bullet,\,\mathrm{Cap}}_d.
\end{aligned}
\end{equation*}
So by Lemma \ref{lem:stationary-cap-rank} again, we have
\[\Delta_{k,d}^{\mathcal C/(Q\cup\{p\})}
\left(
\prod_{i=1}^u\tau_{a_i}(\id)\,
\mathbf B^\circ
\ \middle|\
\boldsymbol{\eta},\lambda
\right).\]

This proves the required vanishing result.
\end{proof}

Recall that
\[
\OrbCap_r=(\PP^1_{r,1},\infty),
\]
where \(0\) is the orbifold point of order \(r\) and \(\infty\) is an
ordinary point.  For \(r=1\), this is the ordinary relative cap, and the relative Virasoro holds by \cite{OP3}.
For \(r\ge2\), Milanov and Tseng proved that the quantum cohomology $\PP^1_{r,1}$ is semisimple \cite{MT1}, and consequently, its absolute Virasoro holds by the Givental-Teleman reconstruction theorem. The following theorem establishes relative Virasoro for $\OrbCap_r$ when $r\ge2$.

\begin{thm}
\label{thm:orbifold-cap2}
For every \(r\ge1\), \(\mu\vdash d\), \(k\ge-1\), we have
\[
L_{k}Z^{\OrbCap_r}_d[\mu]=0.
\]
\end{thm}

\begin{proof}
Degenerate the absolute orbifold line at an ordinary point:
\[
\PP^1_{r,1}
\rightsquigarrow
\OrbCap_r\cup_p\mathrm{Cap},
\]
with the orbifold point lying on \(\OrbCap_r\). Applying Theorem~\ref{thm:cap-inversion} with $\boldsymbol{\eta}=\varnothing$, and we see that the absolute Virasoro for $\PP^1_{r,1}$ implies the relative Virasoro for $\OrbCap_r$. This proves the theorem.
\end{proof}


{\large \part{Extension beyond the effective case}\label{part:gerbe-extension}}


Let $\mathcal{X}$ be a $1$-dimensional smooth proper Deligne-Mumford stack with projective coarse moduli space. Denote by $G$ the group of generic stabilizers of $\mathcal{X}$. By rigidification (see \cite{ACV} and \cite{AGV}), there exists a morphism 
\begin{equation*}
\mathcal{X}\to \mathcal{C}:=\mathcal{X}\thickslash G    
\end{equation*}
which is a $G$-gerbe over an orbifold curve $\mathcal{C}$ (with trivial generic stabilizer). By gerbe decomposition (\cite{HHPS}, \cite{TT1}, \cite{TT3}), the Gromov-Witten theory of $\mathcal{X}$ is equivalent to the Gromov-Witten theory of a disconnected space $\widehat{\mathcal{X}}$ twisted by a $\mathbb{C}^*$-gerbe $\mathbf{c}$. By working with Hodge grading operator and multiplication of $c_1$, it is straightforward to define Virasoro constraints for $\mathbf{c}$-twisted Gromov-Witten theory of $\widehat{\mathcal{X}}$ and show that they are compatible with Virasoro constraints for $\mathcal{X}$. 

It is tempting to study Virasoro constraints for $\mathcal{X}$ via Virasoro constraints for $(\widehat{\mathcal{X}},\mathbf{c})$. Here are some remarks about this:
\begin{enumerate}
\item The equivalence between Gromov-Witten theories of $\mathcal{X}$ and $(\widehat{\mathcal{X}},\mathbf{c})$ is not proven in full generality.

\item Suppose the $G$-gerbe $\mathcal{X}\to \mathcal{C}$ has trivial band. In this case the equivalence of Gromov-Witten theories is proven \cite{AJT}, \cite{AJT2}, \cite{TT2}. Also, in this case $\widehat{\mathcal{X}}=\coprod_{\rho\in \hat{G}}\mathcal{C}_{\rho}$ is the disjoint of $\mathcal{C}_\rho=\mathcal{C}$ indexed by isomorphism classes $\rho$ of irreducible complex representations of $G$. The $\mathbb{C}^*$-gerbe $\mathbf{c}$ can be explicitly constructed: on the component $\mathcal{C}_\rho$, $\mathbf{c}$ is obtained from $[\mathcal{X}\to\mathcal{C}]\in H^2(\mathcal{C},G)$ via $H^2(\mathcal{C},G)\to H^2(\mathcal{C},\mathbb{C}^*)$ induced from the character of $\rho$.

\item In general, $\mathbf{c}$ need not be trivial. $\mathbf{c}$ is trivial exactly when the $G$-gerbe $\mathcal{X}\to\mathcal{C}$ is {\em essentially trivial} (in the sense of \cite{FMN}). 

\item Thus, the results of this paper implies Virasoro constraints for $\mathcal{X}$ when the $G$-gerbe $\mathcal{X}\to\mathcal{C}$ is essentially trivial. When the $G$-gerbe has trivial band, deducing Virasoro constraints for $\mathcal{X}$ requires an extension of results of the present paper to $\mathbf{c}$-twisted Gromov-Witten theory. In principle, this can be done, but requires extra efforts.

\end{enumerate}

\appendix

\section{Auxiliary commutators and left-vacuum evaluations}
\label{app:vacuum}

\begin{lem}\label{lem:U2Uk}
For \(k\ge2\),
\[
  [U_2,U_k]
  =
  \frac{(k+1)(2-k)}2U_{k+1}
  +\frac1r\left(
    (k+1)\sum_{j=2}^{k+1}\frac1j-k-\frac12
  \right)P_k.
\]
\end{lem}
\begin{proof}
Set
\[
  \widehat A_0^{(2)}(w):=[t^2]\widehat A_0(w,t).
\]
Expanding \eqref{eq:P-two-param-comm} through order \(t^2\) gives
\begin{equation}\label{eq:P-A02-comm}
  [P(z),\widehat A_0^{(2)}(w)]
  =
  \frac{zw^2}{z+w}U(z+w)
  +\frac{zw^2}{r}\log\left(1+\frac zw\right)P(z+w).
\end{equation}
Since the untwisted commutator \eqref{eq:johnson-commutator-support} is linear in \(t\), extracting the
\(t^2\)-coefficient from
\[
  [\widehat A_0(z,t),\widehat A_0(w,t)],
  \qquad
  \widehat A_0(z,t)
  =
  P(z)+tU(z)+t^2\widehat A_0^{(2)}(z)+O(t^3),
\]
gives
\[
  [U(z),U(w)]
  =
  [P(w),\widehat A_0^{(2)}(z)]
  -
  [P(z),\widehat A_0^{(2)}(w)].
\]

The coefficient of \(z^3w^{k+1}\) in the \(U\)-part is
\[
  [z^2w^k](z-w)(z+w)^{k+1}U_{k+1}
  =
  \left(
    \binom{k+1}{1}-\binom{k+1}{2}
  \right)U_{k+1},
\]
which is \(\frac{(k+1)(2-k)}2U_{k+1}\).

For the \(P\)-part, its coefficient is
\begin{align*}
  &\frac1r\bigg(
    [x^k](1+x)^{k+1}\log(1+x)-[y^2](1+y)^{k+1}\log(1+y)\bigg)P_k\\
  = &\frac1r\left(
    (k+1)\sum_{j=2}^{k+1}\frac1j-k-\frac12
  \right)P_k.
\end{align*}
Here we have used
\[
  [x^\ell](1+x)^q\log(1+x)
  =
  \left.
    \frac{\partial}{\partial s}\binom{q+s}{\ell}
  \right|_{s=0}
  =
  \binom q\ell\bigl(h_q-h_{q-\ell}\bigr),
\]
with \(h_i\) given by \eqref{eq:harmonic}.
\end{proof}

\begin{lem}\label{lem:nonnegative-stationary-commute}
For \(m,n\geq0\), the operators
\[
  P_m,
  \qquad
  T_{n,a}\quad(1\leq a\leq r-1)
\]
commute pairwise.
\end{lem}
\begin{proof}
By \eqref{eq:johnson-commutator-support}, a nonzero commutator would
require the two coefficient indices to sum to \(-1\), which is impossible
when both are nonnegative.
\end{proof}

\begin{proof}[Proof of Proposition~\ref{prop:vacuum-removal}]
For \(j\in\mathbb Z\) and \(q\geq0\), set
\[
  \Ecal_j^{[q]}:=[z^q]\Ecal_j(z).
\]
For \(n>0\), the Boson--Fermion correspondence gives
\[
  \Ecal_{-n}(z)\vac
  =
  \frac1{\vars(z)}[x^n]
  \exp\left(
    \sum_{m\geq1}\frac{\vars(mz)}m\alpha_{-m}x^m
  \right)\vac.
\]
See, for example, \cite[Chapter~14]{Kac90} for the Boson--Fermion correspondence and \cite[Appendix~A]{Oko01} for the infinite-wedge conventions used here.

Taking the coefficients of \(z^0,z^1,z^2\), respectively, yields
\begin{align}
  \Ecal_{-n}^{[0]}\vac
  &=
  \alpha_{-n}\vac,
  \label{eq:E0exp}\\
  \Ecal_{-n}^{[1]}\vac
  &=
  \frac12\sum_{i+j=n}\alpha_{-i}\alpha_{-j}\vac,
  \label{eq:E1exp}\\
  \Ecal_{-n}^{[2]}\vac
  &=
  \left[
    \frac16\sum_{i+j+\ell=n}
    \alpha_{-i}\alpha_{-j}\alpha_{-\ell}
    +
    \frac{n^2-1}{24}\alpha_{-n}
  \right]\vac,
  \label{eq:E2exp}
\end{align}
where every summation index is positive.

\smallskip
\noindent\emph{The case \(k=1\).}
After taking adjoints, the summands with \(m<0\) annihilate the vacuum.
Formula~\eqref{eq:A0hat} therefore gives
\begin{equation}\label{eq:Ustar-series}
  U^*(z)\vac
  =
  \frac zr\sum_{m\geq0}
  \frac{(z\Scal(rz))^m}{m!}
  \bigl(\log\Scal(rz)-h_m\bigr)
  \Ecal_{-rm}(z)\vac.
\end{equation}
Only \(m=0,1\) contribute to the coefficient of \(z^2\), so
\[
  U_1^*\vac
  =
  \left(-\frac1r\alpha_{-r}+\frac r{24}\right)\vac.
\]
Similarly,
\[
  P_0^*\vac
  =
  \left(H+\alpha_{-r}-\frac1{24}\right)\vac =
  \left(\alpha_{-r}-\frac1{24}\right)\vac.
\]
These two identities give \eqref{eq:U1-vacuum}.

\smallskip
\noindent\emph{The twisted-sector operator and its adjoint.}
For \(1\leq a\leq r-1\), put \(\theta_a=a/r\).  Since
\(T_a(z)=\widehat A_a(z,0)\), formula \eqref{eq:Aahat} gives
\begin{equation}\label{eq:Ta-adjoint-exact}
\begin{aligned}
  T_a(z)
  &=
  \frac za\Scal(rz)^{\theta_a}
  \sum_{m\in\mathbb Z}
  \frac{(z\Scal(rz))^m}{(1+\theta_a)_m}
  \Ecal_{rm+a}(z),\\
  T_a(z)^*
  &=
  \frac za\Scal(rz)^{\theta_a}
  \sum_{m\in\mathbb Z}
  \frac{(z\Scal(rz))^m}{(1+\theta_a)_m}
  \Ecal_{-rm-a}(z).
\end{aligned}
\end{equation}
Only the summands with \(m\geq0\) survive when \(T_a(z)^*\) acts directly
on the vacuum.  For a product, however,
\[
  (T_{p,a}T_{q,a^\vee})^*\vac
  =
  T_{q,a^\vee}^*T_{p,a}^*\vac,
\]
so the summands with \(m<0\) in the left factor may act nontrivially on
the state created by the right factor.  Thus the full Laurent series in
\eqref{eq:Ta-adjoint-exact} must be used.

For the rest of the proof, write
\[
  R_k
  =
  R_k^{\mathrm{lin}}
  +
  R_k^{\mathrm{unt}}
  +
  R_k^{\mathrm{tw}},
\]
where the three terms are, respectively, the linear, untwisted quadratic,
and twisted quadratic summands in \eqref{eq:Rk-def}.

\smallskip
\noindent\emph{The case \(k=2\).}
From
\begin{equation}\label{eq:P-formula}
  P^*(z)\vac
  =
  \sum_{m\ge0}
  \frac{(z\Scal(rz))^m}{m!}\Ecal_{-rm}(z)\vac
\end{equation}
and \eqref{eq:Ustar-series}, one obtains
\[
  P_1^*\vac
  =
  \left(\Ecal_{-r}^{[1]}+\frac12\alpha_{-2r}\right)\vac,
  \qquad
  U_2^*\vac
  =
  -\frac1r\left(
    \Ecal_{-r}^{[1]}+\frac34\alpha_{-2r}
  \right)\vac.
\]
Consequently,
\begin{equation}\label{eq:k2-gap}
  U_2^*\vac+\frac{3}{2r}P_1^*\vac
  =
  \frac1{2r}\Ecal_{-r}^{[1]}\vac.
\end{equation}
For \(r\geq2\), \eqref{eq:Ta-adjoint-exact} gives
\[
  (T_{0,a}T_{0,a^\vee})^*\vac
  =
  \frac1{a(r-a)}
  \alpha_{-a}\alpha_{-(r-a)}\vac.
\]
Indeed, the only additional potentially contributing term is the
\(m=-1\) term; its coefficient of \(z\) vanishes because
\[
  \Scal(rz)^{-a/r}\frac{\vars(az)}{\vars(z)}
\]
is even.  Therefore the twisted part of \(R_2\) satisfies
\[
  (R_2^{\mathrm{tw}})^*\vac
  =
  \frac1{4r}\sum_{a=1}^{r-1}
  \alpha_{-a}\alpha_{-(r-a)}\vac
  =
  \frac1{2r}\Ecal_{-r}^{[1]}\vac.
\]
Together with \eqref{eq:k2-gap}, this proves
\(\langle0|U_2=\langle0|R_2\).  When \(r=1\), the twisted sum and
\(\Ecal_{-1}^{[1]}\vac\) both vanish, so the same conclusion holds.

\smallskip
\noindent\emph{The case \(k=3\).}
Since
\[
  R_3^{\mathrm{lin}}
  =
  -\frac{11}{6r}P_2,
  \qquad
  R_3^{\mathrm{unt}}
  =
  \frac1{12r}P_0^2,
\]
it suffices to compare
\[
  D_3
  :=
  U_3-R_3^{\mathrm{lin}}-R_3^{\mathrm{unt}}
  =
  U_3+\frac{11}{6r}P_2-\frac1{12r}P_0^2
\]
with \(R_3^{\mathrm{tw}}\).  From \eqref{eq:Rk-def},
\[
  R_3^{\mathrm{tw}}\vac
  =
  \frac r{12}\sum_{a=1}^{r-1}
  \left[
    \theta\phi(1+\phi)T_{0,a}T_{1,a^\vee}
    +
    \theta\phi(1+\theta)T_{1,a}T_{0,a^\vee}
  \right]\vac,
\]
where \(\theta=a/r\) and \(\phi=1-\theta=a^\vee/r\).
Set
\[
  c_a:=\frac{a^2-ar-1}{24r}.
\]
The exact adjoint formula \eqref{eq:Ta-adjoint-exact} gives
\begin{align}
  (T_{0,a}T_{1,a^\vee})^*\vac
  ={}&
  \left[
    \frac{\alpha_{-a}\Ecal_{-a^\vee}^{[1]}}{aa^\vee}
    +
    \frac{\alpha_{-(r+a^\vee)}\alpha_{-a}}
         {aa^\vee(1+\phi)}
    +
    \frac{\alpha_{-r}}{a^\vee}
    +
    c_a
  \right]\vac,
  \label{eq:T0a-T1b}\\
  (T_{1,a}T_{0,a^\vee})^*\vac
  ={}&
  \left[
    \frac{\alpha_{-a^\vee}\Ecal_{-a}^{[1]}}{aa^\vee}
    +
    \frac{\alpha_{-a^\vee}\alpha_{-(r+a)}}
         {aa^\vee(1+\theta)}
    +
    \frac{\alpha_{-r}}a
    +
    c_a
  \right]\vac.
  \label{eq:T1a-T0b}
\end{align}
Only the summands indexed by \(m\in\{-2,-1,0,1\}\) contribute to
\eqref{eq:T0a-T1b}--\eqref{eq:T1a-T0b}.  Substituting these formulas into
\(R_3^{\mathrm{tw}}\), and using
\eqref{eq:E0exp}--\eqref{eq:E2exp}, \eqref{eq:Ustar-series}, and the
formula \eqref{eq:P-formula} for \(P^*(z)\vac\), gives
\begin{equation}\label{eq:k3-common-vector}
\begin{aligned}
  D_3^*\vac
  &=
  (R_3^{\mathrm{tw}})^*\vac\\
  &=
  \left[
    \frac5{6r}\Ecal_{-r}^{[2]}
    +\frac1{6r}\Ecal_{-2r}^{[1]}
    -\frac1{12r}\alpha_{-r}^2
    +\frac{11r^2-12r+1}{144r}\alpha_{-r}
    -\frac{(r^2-1)(r^2+6)}{2880r}
  \right]\vac.
\end{aligned}
\end{equation}
The scalar terms use
\[
  \sum_{a=1}^{r-1}a(r-a)
  =
  \frac{r(r^2-1)}6,
  \qquad
  \sum_{a=1}^{r-1}a^2(r-a)^2
  =
  \frac{r(r^2-1)(r^2+1)}{30}.
\]
Thus \(\langle0|U_3=\langle0|R_3\).  If \(r=1\), the twisted sum is
empty and the vector in \eqref{eq:k3-common-vector} vanishes by
\eqref{eq:E1exp}--\eqref{eq:E2exp}.

\smallskip
\noindent\emph{The induction step.}
Assume \(k\geq3\) and
\[
  \langle0|U_2=\langle0|R_2,
  \qquad
  \langle0|U_k=\langle0|R_k.
\]
For any operator \(\mathcal O\),
\begin{align*}
  \langle U_2U_k\mathcal O\rangle
  &=
  \langle R_kR_2\mathcal O\rangle
  +
  \langle[R_2,U_k]\mathcal O\rangle,\\
  \langle U_kU_2\mathcal O\rangle
  &=
  \langle R_2R_k\mathcal O\rangle
  +
  \langle[R_k,U_2]\mathcal O\rangle.
\end{align*}
Every factor in \(R_j\) is one of \(P_m\) or \(T_{n,a}\) with
\(m,n\geq0\).  Lemma~\ref{lem:nonnegative-stationary-commute} therefore
gives \([R_2,R_k]=0\), and hence
\begin{equation}\label{eq:induction-core}
  \langle[U_2,U_k]\mathcal O\rangle
  =
  \left\langle
    \bigl([R_2,U_k]-[R_k,U_2]\bigr)\mathcal O
  \right\rangle.
\end{equation}

We claim that
\begin{equation}\label{eq:R-recursion}
  [R_2,U_k]-[R_k,U_2]
  =
  \gamma_kR_{k+1}+\frac{\delta_k}{r}P_k,
\end{equation}
where
\[
  \gamma_k=\frac{(k+1)(2-k)}2,
  \qquad
  \delta_k=(k+1)\sum_{j=2}^{k+1}\frac1j-k-\frac12.
\]
For the linear terms, \eqref{eq:P-U-coeff} gives
\[
  [R_2^{\mathrm{lin}},U_k]-[R_k^{\mathrm{lin}},U_2]
  =
  \frac{k+1}{2r}(kh_k-3)P_k,
\]
which is the required coefficient because
\[
  -\gamma_kh_{k+1}+\delta_k
  =
  \frac{k+1}{2}(kh_k-3).
\]

For the untwisted quadratic terms, set
\[
  c_{k,p}
  =
  \frac{(p+1)!(k-p-2)!}{k!}
  \quad(0\leq p\leq k-3),
\]
and set \(c_{k,p}=0\) outside this range.  Commuting \(U_2\) with the
two factors gives
\[
  c_{k,p}\binom{k-p-1}{2}
  +
  c_{k,p-1}\binom{p+1}{2}
  =
  \frac{(k+1)(k-2)}2c_{k+1,p},
\]
which is the required coefficient since
\(\gamma_k=-(k+1)(k-2)/2\).

For the twisted quadratic terms, fix \(a\), write
\(\theta=a/r\), and put
\[
  d_{k,p}
  =
  \frac r{k!}(\theta)_{p+1}(1-\theta)_{k-p-1}.
\]
For \(1\leq p\leq k-2\),
\[
  \gamma_kd_{k+1,p}
  =
  -d_{k,p-1}\frac{(p-1+\theta)_2}{2}
  -
  d_{k,p}\frac{(k-p-1-\theta)_2}{2}.
\]
At \(p=0\), after pairing the summands indexed by \(a\) and \(a^\vee\),
the contribution from \([R_2^{\mathrm{tw}},U_k]\) gives
\[
  \gamma_kd_{k+1,0}
  =
  d_{2,0}\frac{(1-\theta)_k}{k!}
  -
  d_{k,0}\frac{(k-1-\theta)_2}{2}.
\]
The identity for \(p=k-1\) follows by exchanging \(\theta\) and
\(1-\theta\).  Substituting the definition of \(d_{k,p}\) verifies these
identities and proves \eqref{eq:R-recursion}.

Finally, Lemma~\ref{lem:U2Uk}, \eqref{eq:induction-core}, and
\eqref{eq:R-recursion} imply
\[
  \gamma_k\langle U_{k+1}\mathcal O\rangle
  +\frac{\delta_k}{r}\langle P_k\mathcal O\rangle
  =
  \gamma_k\langle R_{k+1}\mathcal O\rangle
  +\frac{\delta_k}{r}\langle P_k\mathcal O\rangle.
\]
Since \(\gamma_k\neq0\) for \(k\geq3\),
\[
  \langle0|U_{k+1}=\langle0|R_{k+1}.
\]
Together with the cases \(k=2,3\), this completes the induction and the
proof.
\end{proof}

\bibliography{universal-BIB}

@article {KW23,
    AUTHOR = {Urundolil Kumaran, A. and Wu, L.},
     TITLE = {A new approach to the operator formalism for {G}romov-{W}itten
              invariants of the cap and tube},
   JOURNAL = {Adv. Math.},
  FJOURNAL = {Advances in Mathematics},
    VOLUME = {435},
      YEAR = {2023},
     PAGES = {Paper No. 109357, 49}
}

@misc {AJT2,
   author = {{Andreini}, E. and {Jiang}, Y. and {Tseng}, H.-H.},
    title = "{Gromov--Witten theory of banded gerbes over schemes}",
 year = {2011},
 note = {Preprint, arXiv:1101.5996}
}

@article{AGV,
 author = {Abramovich, D. and Graber, T. and Vistoli, A.},
 title = {Gromov-{Witten} theory of {Deligne}-{Mumford} stacks},
 fjournal = {American Journal of Mathematics},
 journal = {Am. J. Math.},
 issn = {0002-9327},
 volume = {130},
 number = {5},
 pages = {1337--1398},
 year = {2008},
 doi = {10.1353/ajm.0.0017},
 zbMATH = {5363941},
 Zbl = {1193.14070}
}

@article{CR,
  title="{Orbifold Gromov-Witten theory}",
 author={{Chen}, W. and {Ruan}, Y.},
 journal={in: "Orbifolds in mathematics and physics (Madison, WI, 2001)", 25--85, 
Contemp. Math.},
  volume={310},
  year={2002},
  publisher={Amer. Math. Soc., Providence, RI,}
}

@article{FWY,
   author = {{Fan}, H. and {Wu}, L. and {You}, F.},
    title = "{Structures in genus-zero relative Gromov--Witten theory}",
   journal = {J. Topol.},
  fjournal = {Journal of Topology},
    volume = {13},
      year = {2020},
    number = {1},
    pages  = {269--307}
}

@article{FWY2,
  title="{Higher genus relative Gromov–Witten theory and double ramification cycles}",
 author = {{Fan}, H. and {Wu}, L. and {You}, F.},
  journal = {J. London Math. Soc.},
  volume = {103},
  year = 2021,
  pages = {1547-1576},
}

@article {IPsum,
author = {Ionel, E.-N. and Parker, T.-H.},
title = {The symplectic sum formula for {G}romov--{W}itten invariants},
journal = {Ann. of Math. (2)},
fjournal = {Annals of Mathematics. Second Series},
volume = {159},
year = {2004},
number = {3},
pages = {935--1025},
doi = {10.4007/annals.2004.159.935},
}

@article {LR,
author = {Li, A.-M. and Ruan, Y.},
title = {Symplectic surgery and {G}romov-{W}itten invariants of
              {C}alabi-{Y}au 3-folds},
journal = {Invent. Math.},
fjournal = {Inventiones Mathematicae},
volume = {145},
year = {2001},
number = {1},
pages = {151--218},
}

@misc {CLSZ,
author = {Chen, B. and Li, A.-M. and Sun, S. and Zhao, G.},
title = {Relative orbifold {G}romov--{W}itten theory and degeneration formula},
year = {2011},
note = {Preprint, arXiv:1110.6803},
}

@article{OP1,
 author = {Okounkov, A. and Pandharipande, R.},
 title = {Gromov-{Witten} theory, {Hurwitz} theory, and completed cycles},
 fjournal = {Annals of Mathematics. Second Series},
 journal = {Ann. Math. (2)},
 issn = {0003-486X},
 volume = {163},
 number = {2},
 pages = {517--560},
 year = {2006},
 doi = {10.4007/annals.2006.163.517},
 zbMATH = {5049020},
 Zbl = {1105.14076}
}

@article{OP3,
 author = {Okounkov, A. and Pandharipande, R.},
 title = {Virasoro constraints for target curves},
 fjournal = {Inventiones Mathematicae},
 journal = {Invent. Math.},
 issn = {0020-9910},
 volume = {163},
 number = {1},
 pages = {47--108},
 year = {2006},
 doi = {10.1007/s00222-005-0455-y},
 zbMATH = {5000066},
 Zbl = {1140.14047}
}

@article{Tel,
 author = {Teleman, {C}.},
 title = {The structure of {2D} semi-simple field theories},
 fjournal = {Inventiones Mathematicae},
 journal = {Invent. Math.},
 issn = {0020-9910},
 volume = {188},
 number = {3},
 pages = {525--588},
 year = {2012},
 doi = {10.1007/s00222-011-0352-5},
 zbMATH = {6050512},
 Zbl = {1248.53074}
}

@article{JT,
 author = {Jiang, Y. and Tseng, H.-H.},
 title = {On {Virasoro} constraints for orbifold {Gromov}-{Witten} theory},
 fjournal = {IMRN. International Mathematics Research Notices},
 journal = {Int. Math. Res. Not.},
 issn = {1073-7928},
 volume = {2010},
 number = {4},
 pages = {756--781},
 year = {2010},
 doi = {10.1093/imrn/rnp157},
 zbMATH = {5681406},
 Zbl = {1190.14055}
}

@article{TY,
 author = {Tseng, H.-H. and You, F.},
 title = {Higher genus relative and orbifold {Gromov}-{Witten} invariants},
 fjournal = {Geometry \& Topology},
 journal = {Geom. Topol.},
 issn = {1465-3060},
 volume = {24},
 number = {6},
 pages = {2749--2779},
 year = {2020},
 doi = {10.2140/gt.2020.24.2749},
 zbMATH = {7305779},
 Zbl = {1471.14117}
}

@aritle{V,
  author = {{Vakil},R.},
  title = {Counting curves on rational surfaces},
  journal = {Manuscripta Math.},
  Volume = {102},
  Number = {1},
  Pages = {53--84},
  Year = {2000}
}

@article{Z,
	Author = {A. Zinger},
	Journal = {J. Amer. Math. Soc.},
	Number = {3},
	Pages = {691--737},
	Title = {{The reduced genus-one Gromov--Witten invariants of Calabi--Yau hypersurfaces}},
	Volume = {22},
	Year = {2009}}

@article {W,
   author = {{Wu}, L.},
    title = "{A Remark on Gromov--Witten Invariants of Quintic Threefold}",
  journal = {ArXiv e-prints},
archivePrefix = "arXiv",
   eprint = {1705.06402},
 primaryClass = "math.AG",
     year = 2017,
    month = may
}

@article{B,
	Author = {{Bernardara}, M.},
	Journal = {Math. Nachr.},
	Number = {10},
	Pages = {1406--1413},
	Title = {{A semiorthogonal decomposition for Brauer Severi schemes}},
	Volume = {282},
	Year = {2009}}

@article{R,
	Author = {{Ruan}, Y.},
	Journal = {''Northern California Symplectic Geometry Seminar", Amer. Math. Soc. Transl. Ser. 2},
	Pages = {183--198},
	Title = {{Surgery, quantum cohomology and birational geometry}},
	Volume = {196},
	Year = {1999}}

@article{M,
	Author = {{Manolache}, C.},
	Journal = {J. Algebraic Geom.},
	Pages = {201--245},
	Title = {Virtual pull-backs},
	Volume = {21},
	Year = {2012}}

@article{CGT,
	Author = {{Coates}, T. and {Givental}, A. and {Tseng}, H.-H.},
 title = {Virasoro constraints for toric bundles},
 fjournal = {Forum of Mathematics, Pi},
 journal = {Forum Math. Pi},
 issn = {2050-5086},
 volume = {12},
 pages = {28},
 note = {Id/No e4},
 year = {2024},
 doi = {10.1017/fmp.2024.2},
 zbMATH = {7800979},
 Zbl = {1536.14049}
}

@article{Jun2,
	Author = {{Li}, J.},
	Journal = {J. Differential Geom.},
	Pages = {199--293},
	Title = {{A Degeneration formula of GW-invariants}},
	Volume = {60},
	Year = {2002}}

@article{A,
	Author = {M. Atiyah},
	Note = {{\tt arXiv:math/0012213}},
	Title = {K-Theory Past and Present},
	Year = {2000}}

@article{Giv5,
 author = {Givental, A.},
 title = {Gromov-{Witten} invariants and quantization of quadratic {Hamiltonians}},
 fjournal = {Moscow Mathematical Journal},
 journal = {Mosc. Math. J.},
 issn = {1609-3321},
 volume = {1},
 number = {4},
 pages = {551--568},
 year = {2001},
 url = {www.ams.org/distribution/mmj/vol1-4-2001/gwi.pdf},
 zbMATH = {1751354},
 Zbl = {1008.53072}
}

@article{Giv6,
	Author = {A. Givental},
	Journal = {Internat. Math. Res. Notices},
	Number = {23},
	Pages = {1265--1286},
	Title = {Semisimple {F}robenius structures at higher genus},
	Year = {2001}}

@incollection {P,
    AUTHOR = {Pandharipande, R.},
     TITLE = {A calculus for the moduli space of curves},
 BOOKTITLE = {Algebraic geometry: {S}alt {L}ake {C}ity 2015},
    SERIES = {Proc. Sympos. Pure Math.},
    VOLUME = {97},
     PAGES = {459--487},
 PUBLISHER = {Amer. Math. Soc., Providence, RI},
      YEAR = {2018}
}

@article {J,
    AUTHOR = {Johnson, P.},
     TITLE = {Equivariant {GW} theory of stacky curves},
   JOURNAL = {Comm. Math. Phys.},
  FJOURNAL = {Communications in Mathematical Physics},
    VOLUME = {327},
      YEAR = {2014},
    NUMBER = {2},
     PAGES = {333--386}
}

@article {MT1,
    AUTHOR = {Milanov, T. and Tseng, H.-H.},
     TITLE = {The spaces of {L}aurent polynomials, {G}romov-{W}itten theory
              of {$\mathbb{P}^1$}-orbifolds, and integrable hierarchies},
   JOURNAL = {J. Reine Angew. Math.},
  FJOURNAL = {Journal f\"{u}r die Reine und Angewandte Mathematik. [Crelle's
              Journal]},
    VOLUME = {622},
      YEAR = {2008},
     PAGES = {189--235}
}

@article{K,
 author = {Ke, H.-Z.},
 title = {On semisimplicity of quantum cohomology of {{\(\mathbb{P}^1\)}}-orbifolds},
 fjournal = {Journal of Geometry and Physics},
 journal = {J. Geom. Phys.},
 issn = {0393-0440},
 volume = {144},
 pages = {1--14},
 year = {2019},
 doi = {10.1016/j.geomphys.2019.05.007},
 zbMATH = {7107618},
 Zbl = {1441.14182}
}

@article{TT1,
 author = {Tang, X. and Tseng, H.-H.},
 title = {Duality theorems for {\'e}tale gerbes on orbifolds},
 fjournal = {Advances in Mathematics},
 journal = {Adv. Math.},
 issn = {0001-8708},
 volume = {250},
 pages = {496--569},
 year = {2014},
 doi = {10.1016/j.aim.2013.10.002},
 zbMATH = {6284416},
 Zbl = {1300.14058}
}

@article{AJT,
 author = {Andreini, E. and Jiang, Y. and Tseng, H.-H.},
 title = {Gromov-{Witten} theory of root gerbes. {I}: {Structure} of genus 0 moduli spaces},
 fjournal = {Journal of Differential Geometry},
 journal = {J. Differ. Geom.},
 issn = {0022-040X},
 volume = {99},
 number = {1},
 pages = {1--45},
 year = {2015},
 doi = {10.4310/jdg/1418345536},
 zbMATH = {6399558},
 Zbl = {1326.14130}
}

@article{HHPS,
 author = {Hellerman, S. and Henriques, A. and Pantev, T. and Sharpe, E. and Ando, M.},
 title = {Cluster decomposition, {T}-duality, and gerby {CFT}s},
 fjournal = {Advances in Theoretical and Mathematical Physics},
 journal = {Adv. Theor. Math. Phys.},
 issn = {1095-0761},
 volume = {11},
 number = {5},
 pages = {751--818},
 year = {2007},
 doi = {10.4310/ATMP.2007.v11.n5.a2},
 zbMATH = {5248961},
 Zbl = {1156.81039}
}

@article{TT2,
 author = {Tang, X. and Tseng, H.-H.},
 title = {A quantum {Leray}-{Hirsch} theorem for banded gerbes},
 fjournal = {Journal of Differential Geometry},
 journal = {J. Differ. Geom.},
 issn = {0022-040X},
 volume = {119},
 number = {3},
 pages = {459--511},
 year = {2021},
 doi = {10.4310/jdg/1635368578},
 url = {projecteuclid.org/journals/journal-of-differential-geometry/volume-119/issue-3/A-quantum-LerayHirsch-theorem-for-banded-gerbes/10.4310/jdg/1635368578.full},
 zbMATH = {7467769},
 Zbl = {1491.14079}
}

@article{TT3,
 author = {Tang, X. and Tseng, H.-H.},
 title = {On gerbe duality and relative {Gromov}-{Witten} theory},
 fjournal = {Advances in Theoretical and Mathematical Physics},
 journal = {Adv. Theor. Math. Phys.},
 issn = {1095-0761},
 volume = {25},
 number = {8},
 pages = {2171--2178},
 year = {2021},
 doi = {10.4310/ATMP.2021.v25.n8.a5},
 zbMATH = {7590053},
 Zbl = {1514.81221}
}

@article{FMN,
 author = {Fantechi, B. and Mann, E. and Nironi, F.},
 title = {Smooth toric {Deligne}-{Mumford} stacks},
 fjournal = {Journal f{\"u}r die Reine und Angewandte Mathematik},
 journal = {J. Reine Angew. Math.},
 issn = {0075-4102},
 volume = {648},
 pages = {201--244},
 year = {2010},
 doi = {10.1515/CRELLE.2010.084},
 zbMATH = {5865183},
 Zbl = {1211.14009}
}

@article{ACV,
 author = {Abramovich, D. and Corti, A. and Vistoli, A.},
 title = {Twisted bundles and admissible covers},
 fjournal = {Communications in Algebra},
 journal = {Commun. Algebra},
 issn = {0092-7872},
 volume = {31},
 number = {8},
 pages = {3547--3618},
 year = {2003},
 doi = {10.1081/AGB-120022434},
 zbMATH = {1961939},
 Zbl = {1077.14034}
}

@book{Kac90,
 author = {Kac, V. G.},
 title = {Infinite dimensional {Lie} algebras.},
 edition = {3rd ed.},
 isbn = {0-521-37215-1},
 year = {1990},
 publisher = {Cambridge etc.: Cambridge University Press},
 zbMATH = {194085},
 Zbl = {0716.17022}
}

@article {Oko01,
    AUTHOR = {Okounkov, A.},
 title = {Infinite wedge and random partitions},
 fjournal = {Selecta Mathematica. New Series},
 journal = {Sel. Math., New Ser.},
 issn = {1022-1824},
 volume = {7},
 number = {1},
 pages = {57--81},
 year = {2001},
 doi = {10.1007/PL00001398},
 zbMATH = {1684835},
 Zbl = {0986.05102}
}

@article{AF,
  author  = {Abramovich, D. and Fantechi, B.},
  title   = {Orbifold techniques in degeneration formulas},
  journal = {Ann. Sc. Norm. Super. Pisa Cl. Sci. (5)},
  volume  = {16},
  number  = {2},
  pages   = {519--579},
  year    = {2016},
  doi     = {10.2422/2036-2145.201408_006}
}

@misc {Tseng26,
 author = {H.-H. Tseng},
 title = {A note on {Virasoro} constraints for products},
 year = {2026},
 note = {Preprint, arXiv:2603.22486}
}
\bibliographystyle{amsxport}

\end{document}